\documentclass{amsart}
\usepackage{amssymb}
\usepackage{mathrsfs}
\usepackage{stmaryrd}
\usepackage[all]{xy}
\usepackage{graphicx}
\usepackage{tikz}

\newcommand{\Z}{\mathbb{Z}}
\newcommand{\C}{\mathbb{C}}
\newcommand{\bk}{\Bbbk}

\newcommand{\simto}{\overset{\sim}{\to}}
\DeclareMathOperator{\id}{id}
\DeclareMathOperator{\im}{im}
\DeclareMathOperator{\Hom}{Hom}
\DeclareMathOperator{\Ext}{Ext}

\DeclareMathOperator{\Sym}{Sym}
\newcommand{\opp}{\mathrm{opp}}
\newcommand{\parity}{\mathrm{parity}}
\newcommand{\proj}{\mathrm{proj}}
\newcommand{\ngmod}{\mathsf{mod}}
\newcommand{\gmod}{\mathsf{gmod}}
\newcommand{\lh}{\text{-}}
\newcommand{\fg}{\mathrm{fg}}

\newcommand{\cO}{\mathcal{O}}

\newcommand{\fh}{\mathfrak{h}}
\newcommand{\scS}{\mathscr{S}}
\newcommand{\cV}{\mathcal{V}}
\newcommand{\cI}{\mathcal{I}}
\newcommand{\cJ}{\mathcal{J}}
\newcommand{\cB}{\mathcal{B}}
\newcommand{\cP}{\mathcal{P}}

\newcommand{\cE}{\mathcal{E}}
\newcommand{\cF}{\mathcal{F}}
\newcommand{\cG}{\mathcal{G}}
\newcommand{\ubk}{\underline{\bk}}

\newcommand{\Sh}{\mathsf{Sh}}
\newcommand{\She}{\Sh}

\newcommand{\For}{\mathsf{For}}
\newcommand{\Parity}{\mathsf{Parity}}
\newcommand{\BMP}{\mathsf{BMP}}
\newcommand{\BMPe}{\BMP}
\newcommand{\BMPc}{\overline{\BMP}}
\newcommand{\cZ}{\mathcal{Z}}
\newcommand{\cC}{\mathcal{C}}
\newcommand{\cD}{\mathcal{D}}
\newcommand{\Cref}{\cC_\mathrm{ref}}

\newcommand{\SBim}{\mathsf{SBim}}

\newcommand{\dashvdl}{\mathbin{\rotatebox[origin=c]{45}{$\dashv$}}}
\newcommand{\dashvdr}{\mathbin{\rotatebox[origin=c]{-45}{$\dashv$}}}

\newcommand{\Kb}{K^{\mathrm{b}}}
\newcommand{\Db}{D^{\mathrm{b}}}
\newcommand{\sD}{\mathsf{D}}

\newcommand{\mix}{\mathrm{mix}}
\newcommand{\Dmix}{D^\mix}
\newcommand{\Dmixe}{\Dmix}
\newcommand{\Dmixc}{\Dmix_c}
\newcommand{\la}{\langle}
\newcommand{\ra}{\rangle}

\newcommand{\p}{{}^p\!}

\newcommand{\Pmix}{\mathsf{P}^\mix}
\newcommand{\Pmixe}{\Pmix}
\newcommand{\Pmixc}{\Pmix_c}
\newcommand{\IC}{\mathcal{IC}}
\newcommand{\Tmixc}{\mathsf{Tilt}^\mix_c}
\newcommand{\cT}{\mathcal{T}}

\newtheorem*{thm*}{Theorem}
\numberwithin{equation}{section}
\newtheorem{thm}{Theorem}[section]
\newtheorem{lem}[thm]{Lemma}
\newtheorem{prop}[thm]{Proposition}
\newtheorem{cor}[thm]{Corollary}
\theoremstyle{definition}
\newtheorem{defn}[thm]{Definition}
\theoremstyle{remark}
\newtheorem{rmk}[thm]{Remark}
\newtheorem{ex}[thm]{Example}

\title{Mixed modular perverse sheaves on moment graphs}
\author{Shotaro Makisumi}
\date{\today}
\address{Department of Mathematics, Stanford University, Stanford, CA, U.S.A.}
\email{makisumi@stanford.edu}

\begin{document}

\begin{abstract}
 We study an analogue of the Achar--Riche ``mixed modular derived category''~\cite{AR-II} for moment graphs. In particular, given a Coxeter group $W$ and a reflection faithful representation $\fh$, we introduce a category that plays the role of Schubert-stratified mixed modular perverse sheaves on ``the flag variety associated to $(W, \fh)$.'' We show that this Soergel-theoretic generalization of graded category $\cO$ is graded highest weight.
\end{abstract}

\maketitle

\section{Introduction}
\label{s:intro}
Motivated by questions in the modular representation theory of reductive groups, Achar and Riche~\cite{AR-II} introduced the ``mixed modular derived category'' of a complex variety with a fixed affine stratification. This is a triangulated category with a notion of ``weights'' and ``Tate twist'' that, for positive characteristic coefficients, may serve as a substitute for Deligne's theory of mixed $\ell$-adic sheaves.

In this paper, we carry out an analogous study for moment graph, a combinatorial gadget with connections to both torus-equivariant geometry and representation theory. We show that the Achar--Riche construction of recollement structure also works in this setting, leading to a ``perverse'' t-structure on the mixed derived category of arbitrary moment graphs, not necessarily arising from geometry.

In particular, given a Coxeter system $(W, S)$ and a suitable ``reflection faithful realization'' $\fh$, this allows us define an abelian category $\Pmixc(\cB)$ of ``constructible mixed perverse sheaves'' on the Bruhat moment graph $\cB$ associated to $(W, \fh)$. This should be viewed as the category of Schubert-stratified mixed modular perverse sheaves on a possibly non-existent flag variety---a Soergel-theoretic generalization of the graded category $\cO$ of~\cite{Soe90, BGS}. One of our main results is that $\Pmixc(\cB)$ is graded highest weight (see Theorem~\ref{thm:ghw}), generalizing the result of~\cite{AR-II} for flag varieties. This lays the groundwork for~\cite{Mak}, in which we prove a Koszul duality in this setting for finite $W$.

\subsection{Motivation from torus-equivariant geometry}
Let $X$ be a complex variety with a fixed affine stratification $\scS$. Let $\Db_\scS(X, \bk)$ be the $\scS$-constructible derived category of sheaves of $\bk$-vector spaces (in this introduction, we work with field coefficients to simplify the exposition). Let also $\Parity_\scS(X, \bk)$ be the full additive subcategory of parity sheaves of Juteau--Mautner--Williamson~\cite{JMW}. Achar--Riche defined~\cite{AR-II} the \emph{mixed derived category} of $X$ with coefficients in $\bk$ as
\[
 \Dmix_\scS(X, \bk) := \Kb\Parity_\scS(X, \bk).
\]
See~\cite{AR-III} for a further study and~\cite{AR} for a related construction.

Meanwhile, given a complex variety $X$ with a suitably nice action of a complex torus $T$ and a $T$-invariant stratification by affine spaces, one may associate a combinatorial gadget $\cG_{X,T}$ called the \emph{moment graph}, a labeled graph encoding the $0$- and $1$-dimensional $T$-orbits~\cite{GKM}. Braden--MacPherson~\cite{BMP} showed that the complex $T$-equivariant intersection cohomology of stratum closures may be computed from certain combinatorially-defined sheaves on $\cG_{X,T}$, nowadays called \emph{BMP sheaves}.

This result was generalized to positive characteristic coefficients $\bk$ by Fiebig--Williamson \cite{FW}. Under further assumptions on the $T$-action, they established an equivalence
\[
 \Parity_T(X, \bk) \cong \BMPe(\cG_{X,T}^\bk)
\]
between $T$-equivariant parity sheaves and BMP sheaves.

Motivated by this connection, we define the \emph{(equivariant) mixed derived category}
\[
 \Dmixe(\cG) := \Kb\BMPe(\cG)
\]
for any moment graph $\cG$, not necessarily arising from geometry. We will show that much of the theory of~\cite{AR-II} can also be developed in this setting.

In particular, we will define ``standard'' and ``costandard'' sheaves associated to each ``stratum'' of $\cG$, as well as a ``perverse'' t-structure on $\Dmixe(\cG)$, with heart $\Pmixe(\cG)$ consisting of ``equivariant mixed perverse sheaves.'' Under a further assumption, we obtain the same results for the ``constructible'' versions $\Dmixc(\cG)$ and $\Pmixc(\cG)$.

\begin{rmk}
 There is a parallel story relating $T$-equivariant sheaves on a toric variety $X$ and sheaves on the associated fan. The analogue of BMP sheaves are called pure sheaves, and Braden--Lunts~\cite{BL} showed that their bounded homotopy category is a mixed version of $\Db_T(X, \C)$. Although we will not explicitly address this setting, the proof of recollement also goes through for (not necessarily rational) fans.
\end{rmk}

\subsection{Motivation from representation theory}
\subsubsection{Relation to Soergel bimodules}
Let $G$ be a connected complex reductive group with Borel subgroup $B$ and maximal torus $T$. Moment graphs of particular interest in representation theory are the \emph{Bruhat (moment) graphs} $\cB$ associated to the left action of $T$ on the flag variety $G/B$ with the Schubert stratification.

Objects related to Bruhat graphs play a key role in Soergel's approach for relating representation theory and geometry. In particular, Soergel gave a new proof of the Kazhdan--Lusztig conjecture by relating the BGG category $\cO$ to the geometry of flag varieties via an intermediate category of ``special modules''~\cite{Soe90}. Fiebig showed that the equivariant analogue of these modules---studied in~\cite{Soe92, Soe00, Soe07} and nowadays called \emph{Soergel bimodules}---are equivalent to BMP sheaves on the Bruhat graph~\cite{Fie-verma, Fie-coxeter}.

This offers an alternative, more geometric point of view on Soergel bimodules. In particular, this allowed Fiebig to give an explicit upper bound on the exceptional primes for Lusztig's conjecture on modular representations of reductive groups~\cite{Fie12}. For an overview of moment graph techniques in representation theory, see~\cite{Fie13}.

\subsubsection{``Geometry'' associated to arbitrary Coxeter groups}
The Bruhat graph is a combinatorial shadow of the flag variety. In fact, like Soergel bimodules, it can be defined for an arbitrary Coxeter group $W$ and a suitable ``reflection faithful representation'' $\fh$. Associated to the datum $(W, \fh)$ are two additive categories: BMP sheaves $\BMPe(\cB)$ on the associated Bruhat graph $\cB$, and Soergel bimodules $\SBim$. Fiebig established an equivalence
\[
 \BMPe(\cB) \cong \SBim
\]
in this generality (see~\S\ref{ss:bmp-sbim}).

In our terminology, we therefore have $\Dmixe(\cB) \cong \Kb\SBim$. With help from Soergel theory, many of the results of~\cite{AR-II} for flag varieties generalize to $\cB$ with the same proof. In particular, we show that $\Pmixc(\cB)$ is graded highest weight (see Theorem~\ref{thm:ghw}). This is a Soergel-theoretic generalization of the graded category $\cO$ of~\cite{Soe90, BGS}.

The idea of viewing $\Kb\SBim$ as the mixed equivariant derived category of the flag variety goes back to Soergel.\footnote{``Komplexe von Bimoduln sind die Gewichtsfiltrierung des armen Mannes'' (quoted in~\cite{WW}). This is made more precise in~\cite{Sch}.} Although no such flag variety exists for general $(W, \fh)$, Elias--Williamson~\cite{EW-hodge} have shown that for any $W$ there exists characteric-zero $\fh$ for which Soergel bimodules behave as though they were equivariant intersection cohomology of actual Schubert varieties---they satisfy hard Lefschetz and the Hodge--Riemann relations---and used this to prove Soergel's conjecture, a Soergel bimodule analogue of the Kazhdan--Lusztig conjecture.

If $W$ is finite and $(W, \fh)$ satisfies Soergel's conjecture, then $\Pmixc(\cB)$ is even Koszul. This motivates one to ask whether $\Pmixc(\cB)$ is Koszul self-dual, generalizing the result of Beilinson--Ginzburg--Soergel~\cite{Soe90, BGS} for category $\cO$. We prove this in~\cite{Mak}. This depends on the framework of the present paper outside the case of $(W, \fh)$ arising from complex reductive groups.

\subsubsection{The 2-braid group and Rouquier complexes}
The category $\Kb\SBim$ was considered earlier in the literature as the ``2-braid group'' by Rouquier~\cite{Rou}, who used certain complexes of Soergel bimodules to categorify the braid group, and has been studied further in~\cite{LW, Jen}. Rouquier complexes play a key technical role in~\cite{EW-hodge} (as a ``weak Lefschetz substitute'') and in the further positivity results of~\cite{Gob}. Particular interest in type A Rouquier complexes arises from their role in Khovanov's definition of triply-graded (HOMFLY) knot homology~\cite{Kho}; see for example~\cite{EK, ARo, Hog, AH, EH, GNR}.

Via the interpretation as $\Dmixe(\cB)$, the present work allows for a more geometric study of the 2-braid group for an arbitrary Coxeter group. In particular, the combinatorial sheaf functors of~\S\ref{s:recollement} allows one to view Rouquier complexes associated to reduced expressions as standard and costandard sheaves (see Theorem~\ref{thm:std-rouquier}). This framework leads to a new, more conceptual proof of two known results: the braid relation for Rouquier complexes~\cite{Rou}; and the main result (``Rouquier's formula'') of~\cite{LW}, a Hom vanishing conjectured in~\cite{Rou-ICM}.

\subsection{Perspective: the Hecke category}
Borel-equivariant parity sheaves on flag varieties, BMP sheaves on Bruhat graphs, and Soergel bimodules are all incarnations of a graded monoidal additive category known as the Hecke category. In fact, Elias and Williamson~\cite{EW-soergel-calculus} (building on~\cite{EK, Eli-dihedral}) defined the diagrammatic Hecke category more generally for a ``realization'' of a Coxeter system over an arbitrary commutative ring.

In this paper, we crucially use the equivalence $\BMP(\cB) \cong \SBim$ to establish the desired properties of $\Dmixc(\cB)$ and $\Pmixc(\cB)$. Since Soergel bimodule theory has only been worked out over a field (and $\SBim$ is in any case ill-behaved for a general realization), this imposes what we believe are unnecessary restrictions on the realization. One expects a more direct relation between the diagrammatic Hecke category and Bruhat BMP sheaves, which would lead to a generalization of the results of the present paper as well as the Koszul duality of~\cite{Mak}.

\subsection{Contents}
In~\S\ref{s:prelim} we fix conventions and recall basic notions about moment graph sheaves and BMP sheaves. In~\S\ref{s:recollement} we define the mixed derived category of a moment graph and show that it admits a recollement structure. In~\S\ref{s:mixed-perverse-sheaves} we define standard and costandard sheaves and the perverse t-structure. In~\S\ref{s:bruhat}, we apply the general theory of~\S\ref{s:recollement}--\ref{s:mixed-perverse-sheaves} to Bruhat graphs. Applications to Rouquier complexes appear in~\S\ref{ss:rouquier-complexes}.

We imitate the development in~\cite{AR-II} for both the general theory and the study of Bruhat graphs. In each part, while the first results often require proofs specific to BMP sheaves, once enough basic results have been proved, the rest of the development can often be copied from~\cite{AR-II} with only notation changes. In such cases, we often omit the proof without comment.

\subsection{Acknowledgements}
This work was started during the Fall 2014 semester on Geometric Representation Theory at the Mathematical Sciences Research Institute. I owe much to P.~Fiebig for his guidance during that semester and his continued encouragement since. I also thank P.~Achar, B.~Elias, S.~Riche, and G.~Williamson for useful discussions.

This work is part of the author's Ph.D.~thesis. I thank my advisor, Z.~Yun, for travel support through his Packard Foundation fellowship, and for countless helpful conversations.

\section{Preliminaries}
\label{s:prelim}
\subsection{Conventions}
Throughout, $\bk$ denotes a commutative ring (of coefficients). All gradings are $\Z$-gradings. For a graded $\bk$-algebra $A$, all homomorphisms between graded $A$-modules are assumed to be of degree $0$. Let $A\lh\gmod$ denote the category of graded $A$-modules and degree 0 homomorphisms. This category has a grading shift autoequivalence $\{1\}$: given $M = \bigoplus M^i \in A\lh\gmod$, we set $M\{n\}^i = M^{i+n}$. For $M, N \in A\lh\gmod$, define the \emph{graded Hom}
\[
 \Hom_{A\lh\gmod}^\bullet(M, N) := \bigoplus_{n \in \Z}\Hom_{A\lh\gmod}(M, N\{n\}).
\]
If $A$ is commutative, this is again a graded $A$-module.

The axioms of recollement~\cite[\S1.4.3]{BBD} have a duality: replacing each category with its opposite category exchanges $i^*$ with $i^!$ and $j_*$ with $j_!$. Recollement thus consists of two dual halves: ``left recollement'' involving $j_!$ and $i^*$, and ``right recollement'' involving $j_*$ and $i^!$. Whenever we say duality in~\S\ref{s:recollement}, we mean this rather than Verdier duality, which we have not developed for general moment graphs.

\subsection{Background on moment graphs}
\label{ss:moment-graph-sheaves}
In this subsection and the next, we review basic notions on moment graphs, largely following~\cite{Fie-verma, Fie-coxeter, FW} but with slight differences.
\begin{defn} \label{defn:moment-graph}
 Let $\fh$ be a free $\bk$-module of finite rank, and let $\fh^* = \Hom_\bk(\fh, \bk)$. An (unordered) \emph{moment graph over $\fh$} consists of the following datum:
 \begin{itemize}
  \item a graph $(\cV, \cE)$ with vertex set $\cV$ and edge set $\cE$;
  \item a map $\alpha \colon \cE \to V^* \setminus \{ 0 \}$ of \emph{edge labels}.
 \end{itemize}
 We assume that any two vertices are connected by at most one edge.
\end{defn}

We will often simply speak of a moment graph $X$ over $\bk$, or even just moment graph $X$, with the other data implied. We write $\cV(X)$ (resp.~$\cE(X)$) for the vertex set (resp.~edge set) of the graph underlying $X$. Any subset $\cI \subset \cV(X)$ defines a \emph{sub-moment graph} of $X$ (over $\fh$) whose graph is the full subgraph corresponding to $\cI$ and whose map of edge labels is obtained by restricting $\alpha$. We will often confuse a set of vertices with the corresponding sub-moment graph.

We write $E\colon x\text{---}y$ to denote an edge connecting vertices $x$ and $y$. A \emph{finite} moment graph $X$ is one for which $\cV(X)$ is finite.

Let $R = \Sym_\bk(V^*)$ the symmetric algebra of $V^*$ over $\bk$, graded so that $\deg(V^*) = 2$.
\begin{defn}
 A \emph{sheaf $\cF$ on a moment graph $(\cV, \cE, \alpha)$ over $\fh$} consists of the following datum:
 \begin{itemize}
  \item a graded $R$-module $\cF^x$ for each $x \in \cV$;
  \item a graded $R$-module $\cF^E$ for each $E \in \cE$, satisfying $\alpha(E)\cF^E = 0$;
  \item a graded $R$-module homomorphism $\rho_{x, E} \colon \cF^x \to \cF^E$ for each pair $(x, E) \in \cV \times \cE$ such that $x$ lies on $E$.
 \end{itemize}
\end{defn}

\begin{rmk}
 In~\cite{Fie-verma, Fie-coxeter, FW}, one instead considers moment graphs $X$ over a free $\Z$-module $\Lambda$, and $\bk$-sheaves on such $X$. In our set-up, these are viewed as sheaves on the moment graph $\bk \otimes_\Z X$ over $\bk \otimes_\Z \Lambda$, obtained from $X$ by base change. Since we do not consider base change in this paper, we will not spell out this notion.
\end{rmk}

Let $\cF, \cG$ be sheaves on a moment graph $X$. A \emph{morphism} $\varphi \colon \cF \to \cG$ consists of graded $R$-module homomorphisms $\varphi^x \colon \cF^x \to \cG^x$ and $\varphi^E \colon \cF^E \to \cG^E$ for each vertex $x$ and edge $E$, compatible with the $\rho$ maps in the obvious way. Denote by $\She(X)$ the category of sheaves on $X$. This category is naturally $\bk$-linear abelian, and has a grading shift autoequivalence $\{1\}$: we define $\cF\{n\}$ by $\cF\{n\}^x = \cF^x\{n\}$, $\cF\{n\}^E = \cF^E\{n\}$, and shifted $\rho$ maps. As with graded modules, we define the \emph{graded Hom}
\[
 \Hom_{\She(X)}^\bullet(\cF, \cG) := \bigoplus_{n \in \Z}\Hom_{\She(X)}(\cF, \cG\{n\}),
\]
which is a graded $R$-module.

For any $\cI \subset \cV(X)$, define the \emph{space of sections of $\cF$ over $\cI$} by
\[
 \Gamma(\cI, \cF) := \left\{ (m_x) \in \prod_{x \in \cI} \cF^x \,\left|\,
                            \begin{matrix} \rho_{x, E}(m_x) = \rho_{y, E}(m_y) \\ \text{ for any edge } E \colon x\text{ --- }y \text{ with } x, y \in \cI \end{matrix}
                            \right.
                     \right\}.
\]
Write $\Gamma(\cF)$ for the \emph{space of global sections} $\Gamma(X, \cF) = \Gamma(\cV(X), \cF)$. For any subsets $\cI \subset \cJ$ of $\cV(X)$, there is an obvious restriction map $\Gamma(\cJ, \cF) \to \Gamma(\cI, \cF)$. In particular, there is a map $\Gamma(\cI, \cF) \to \cF^x$ whenever $x \in \cI$.

\subsection{Background on Braden--MacPherson sheaves}
\label{ss:bmp}
We assume from now on that the vertex set $\cV$ of our moment graph is equipped with a partial order $\le$ such that, for any edge $E\colon x\text{ --- }y$, either $x \le y$ or $y \le x$ (but $x \neq y$). This datum is called an \emph{ordered moment graph}. This induces a topology on $\cV$ in which the open sets are the upward closed sets, i.e.~subsets $\cI$ such that if $x \in \cI$ and $x \le y$, then $y \in \cI$. For any vertex $x$, write $\{ > x \}$ for the set $\{ y \in \cV \mid y > x \}$, and similarly for $\{ \ge x \}$, $\{ < x \}$, and $\{ \le x\}$.

We are interested in the following class of moment graph sheaves.
\begin{defn}
 A sheaf $\cF$ on a moment graph $(\cV, \cE, \alpha)$ is called \emph{BMP (Braden--MacPherson)} if it satisfies the following properties:
 \begin{enumerate}
  \item[(BMP1)] for any $x \in \cV$, $\cF^x$ is a graded free $R$-module of finite rank, and it is nonzero for only finitely many $x$;
  \item[(BMP2)] for any edge $E \colon x \text{ --- } y$ with $x \le y$, the map $\rho_{y, E} \colon \cF^y \to \cF^E$ is surjective with kernel $\alpha(E)\cF^y$;
  \item[(BMP3)] for any open subset $\cJ \subset \cV$, the restriction map $\Gamma(\cF) \to \Gamma(\cJ, \cF)$ is surjective;
  \item[(BMP4)] for any $x \in \cV$, the restriction map $\Gamma(\cF) \to \cF^x$ is surjective.
 \end{enumerate}
\end{defn}

We may replace the last two conditions with the following:
\begin{enumerate}
 \item[(BMP3')] for any $x \in \cV$, $\Gamma(\{ \ge x\}, \cF) \to \Gamma(\{ >x \}, \cF)$ is surjective;
 \item[(BMP4')] for any $x \in \cV$, $\Gamma(\{ \ge x\}, \cF) \to \cF^x$ is surjective.
\end{enumerate}
It will be convenient to introduce some notation. For each vertex $x$, define
\[
 \cV_{\delta x} := \{ y \in \cV \mid \text{there exists an edge } E \colon x\text{---}y \text{ with } x \le y \}
\]
with corresponding edges $\cE_{\delta x}$. For a sheaf $\cF$ and a vertex $x$, define
\[
 u_x\colon \Gamma(\{ >x \}, \cF) \hookrightarrow \bigoplus_{y > x} \cF^y \xrightarrow{p} \bigoplus_{y \in \cV_{\delta x}} \cF^y \xrightarrow{\bigoplus \rho_{y, E}} \bigoplus_{E \in \cE_{\delta x}} \cF^E,
\]
where $p$ is the projection, and
\[
 d_x := (\rho_{x, E})_{E \in \cE_{\delta x}} \colon \cF^x \to \bigoplus_{E \in \cE_{\delta x}} \cF^E.
\]
Thus (BMP3') and (BMP4') say that $\im(u_x) = \im(d_x)$.\footnote{This was original definition of Braden and MacPherson~\cite[Remark after Corollary~1.3]{BMP}.}

Let $\BMPe(X) \subset \She(X)$ denote the full additive subcategory of BMP sheaves. For local ring coefficients, there is a classification theorem.
\begin{thm}[\cite{FW}, Theorem~6.4] \label{thm:bmp-classification}
 Let $X$ be a moment graph over a local ring $\bk$. Assume that $\{ \le x \}$ is finite for any $x \in \cV(X)$. Then every BMP sheaf on $X$ is isomorphic to a finite direct sum of shifts of indecomposable BMP sheaves $\cE_x$ for $x \in \cV(X)$, which are characterized by the following properties:
 \begin{itemize}
  \item $(\cE_x)^y = 0$ unless $y \le x$ (support condition), and $(\cE_x)^x \cong R$ (normalization);
  \item $\cE_x$ is indecomposable in $\She(X)$.
 \end{itemize}
\end{thm}
Moreover, this direct sum decomposition is unique in the obvious sense, which we omit.

\subsection{Notation and terminology}
Throughout this paper, we will use notation and terminology to emphasize the analogy with geometry. For example, a moment graph will be denoted by $X, Y$, etc. Given a moment graph $X$, the sub-moment graph consisting of a single vertex $s$ will be denoted by $X_s$ and called a ``stratum.'' By a ``closed inclusion'' $i\colon Z \hookrightarrow X$ or ``closed union of strata'' $Z$, we mean that $Z$ is a sub-moment graph given by a closed subset of $\cV(X)$. Similar remarks apply for ``open inclusion'' $j\colon U \hookrightarrow X$ and ``locally closed inclusion'' $h \colon Y \hookrightarrow X$. Given a closed inclusion, its ``open complement'' has the obvious meaning.

Moment graph sheaves will be denoted by $\cE, \cF, \cG$, etc. A sheaf $\cF$ on $X$ is said to be \emph{supported} on a closed sub-moment graph $Z$ if $\cF^s = 0$ and $\cF^E = 0$ for all $s \notin \cV(Z)$ and $E \notin \cE(Z)$.

\section{Recollement on moment graphs}
\label{s:recollement}

Unlike the derived category of sheaves on a topological space, the category of sheaves on a moment graph is abelian and has neither the full assortment of sheaf functors nor distinguished triangles. We will show (\S\ref{ss:naive-functors}), however, that it has a ``partial recollement'' structure, and that BMP sheaves satisfy ``distinguished triangle'' short exact sequences of Homs analogous to those of parity sheaves. These straightforward observations are in fact exactly what is needed to carry out the Achar--Riche construction of recollement (\S\ref{ss:recollement}).

Much of the idea in~\S\ref{ss:naive-functors} already appears in the work of Fiebig and (in the toric setting) Braden--Lunts; see Remarks~\ref{rmk:naive-functors-Fiebig} and \ref{rmk:naive-functors-fans}.

\subsection{Naive sheaf functors}
\label{ss:naive-functors}

In this subsection, let $X$ be a finite moment graph, $i\colon Z \hookrightarrow X$ a closed inclusion, and $j\colon U \hookrightarrow X$ its open complement.

\begin{prop} \label{prop:naive-partial-recollement}
 There exist $\bk$-linear functors
 \[
  \xymatrix@C=1.5cm{
   \She(Z) \ar[r]|{i_*} &
   \She(X) \ar[r]|{j^*} \ar@/_1pc/[l]_{i^{[*]}} \ar@/^1pc/[l]^{i^{[!]}} &
   \She(U)
  }
 \]
 such that $i_*$ is a fully faithful, $j^*$ is essentially surjective, $j^*i_* = 0$, $i^{[*]}i_* \cong \id$, $i^{[!]}i_* \cong \id$, and $i^{[*]} \dashv i_* \dashv i^{[!]}$. 
\end{prop}

As usual, set $i_! = i_*$ and $j^! = j^*$. We write $i^{[*]}$ and $i^{[!]}$ to emphasize that these naive restriction and corestriction functors do not have all the usual properties of sheaf functors from geometry.

\begin{proof}
 Let $\cF \in \She(Z)$. Define $i_*\cF \in \She(X)$ as extension by 0:
 \[
 \begin{gathered}
  (i_*\cF)^s = \begin{cases} \cF^s &\mbox{if } s \in \cV(Z) \\ 
                              0     &\mbox{otherwise,}
                \end{cases} \qquad
  (i_*\cF)^E = \begin{cases} \cF^E &\mbox{if } E \in \cE(Z) \\ 
                              0     &\mbox{otherwise,}
                \end{cases} \\
  \rho^{i_*\cF}_{s, E} =
   \begin{cases}
    \rho^\cF_{s, E} &\mbox{if } s \in \cV(Z) \mbox{ and } E \in \cE(Z) \\ 
    0               &\mbox{otherwise.}
   \end{cases}
 \end{gathered}
 \]
 Let $\cF \in \She(X)$. Define $j^*\cF \in \She(U)$ and $i^{[*]}\cF \in \She(Z)$ as restriction:
 \[
  (j^*\cF)^s = \cF^s, \qquad (j^*\cF)^E = \cF^E, \qquad \rho^{j^*\cF}_{s, E} = \rho^\cF_{s, E}
 \]
 for all relevant $s \in \cV(U)$ and $E \in \cE(U)$, and similarly for $i^{[*]}\cF$. Finally, define
 \begin{align*}
  (i^{[!]}\cF)^s &= \{ m \in \cF^s \mid \rho^\cF_{s, E}(m) = 0 \mbox{ for all } E \in \cE_{\delta s} \setminus \cE(Z) \}, \\
  (i^{[!]}\cF)^E &= \cF^E, \qquad \rho^{i^{[!]}\cF}_{s, E} = \rho^\cF_{s, E}|_{(i^{[!]}\cF)^s}.
 \end{align*}
 
 These clearly extend to functors $i_* = i_!$, $j^* = j^!$, $i^{[*]}$, and $i^{[!]}$. All statements except the adjunctions are clear. For the adjunctions, it suffices to show that for any $\cF \in \She(Z)$ and $\cG \in \She(X)$, the maps
 \begin{align*}
  \Hom_{\She(X)}(\cG, i_*\cF) &\to \Hom_{\She(Z)}(i^{[*]}\cG, i^{[*]}i_*\cF) \cong \Hom_{\She(Z)}(i^{[*]}\cG, \cF) \\
  \Hom_{\She(X)}(i_!\cF, \cG) &\to \Hom_{\She(Z)}(i^{[!]}i_!\cF, i^{[!]}\cG) \cong \Hom_{\She(Z)}(\cF, i^{[!]}\cG)
 \end{align*}
 induced by $i^{[*]}$ and $i^{[!]}$, respectively, are isomorphisms. Both are easily checked.
\end{proof}
For convenience, we describe the unit $\varepsilon\colon \cF \to i_*i^{[*]}\cF$ and counit $\eta\colon i_*i^{[!]}\cF \to \cF$ concretely. Since $i_*i^{[*]}\cF$ and $i_*i^{[!]}\cF$ are supported on $Z$, the maps $\varepsilon^s$, $\varepsilon^E$, $\eta^s$, $\eta^E$ are necessarily zero for $s \notin \cV(Z)$ and $E \notin \cE(Z)$. For $s \in \cV(Z)$ and $E \in \cE(Z)$, we have
\begin{align*}
 \varepsilon^s&\colon \cF^s \to \cF^s = (i_*i^{[*]}\cF)^s, &\varepsilon^E\colon \cF^E \to \cF^E = (i_*i^{[*]}\cF)^E, \\
 \eta^s&\colon (i_*i^{[!]}\cF)^s = (i^{[!]}\cF)^s \hookrightarrow \cF^s, &\eta^E\colon (i_*i^{[!]}\cF)^E = \cF^E \to \cF^E,
\end{align*}
where $\eta^s$ is the inclusion, and $\eta^E, \varepsilon^s, \varepsilon^E$ are the identity maps. Note that for any edge $E \in \cE_{\delta s} \setminus \cE(Z)$, the diagram
\[
 \xymatrix{
  0 \ar[r]^-{\eta^E = 0} & \cF^E \\
  (i_*i^{[!]}\cF)^s \ar@{^(->}[r]_-{\eta^s} \ar[u]^{\rho^{i_*i^{[!]}\cF}_{s, E}} & \cF^s \ar[u]_{\rho^\cF_{s, E}}
 }
\]
commutes by the definition of $(i^!\cF)^s$.

\begin{prop} \label{prop:BMP-pseudo-dt}
 For $\cE, \cF \in \BMPe(X)$, we have the short exact sequences of $\bk$-modules
 \begin{align*}
  0 \to \Hom_{\She(X)}(i_*i^{[*]}\cE, \cF) \xrightarrow{\alpha} \Hom_{\She(X)}(\cE, \cF) &\xrightarrow{\gamma} \Hom_{\She(U)}(j^*\cE, j^*\cF) \to 0, \\
  0 \to \Hom_{\She(X)}(\cE, i_*i^{[!]}\cF) \xrightarrow{\beta} \Hom_{\She(X)}(\cE, \cF) &\xrightarrow{\gamma} \Hom_{\She(U)}(j^*\cE, j^*\cF) \to 0,
 \end{align*}
 where $\alpha$ (resp.~$\beta$) is composition with $\varepsilon$ (resp.~$\eta$), and $\gamma$ is induced $j^*$.
\end{prop}

\begin{rmk}
 In the case of parity sheaves, these sequences are obtained by applying $\Hom(-, \cF)$ to the distinguished triangle $j_!j^*\cE \to \cE \to i_*i^*\cE \to$ and using adjunction and the fact that odd $\Ext$ of parity sheaves vanish (or dually, by applying $\Hom(\cE, -)$ to $i_*i^!\cF \to \cF \to j_*j^*\cF \to$).
\end{rmk}

The surjectivity already appears in the proof of~\cite[Proposition~5.1]{Fie-verma}. We repeat the proof for convenience.
\begin{proof}
 The adjunction isomorphism between the first terms is easily seen to identify $\alpha$ and $\beta$, so it suffices to check the first sequence. Clearly $\gamma \circ \alpha = 0$, and $\alpha$ is injective since $i^{[*]}\varepsilon$ is the identity map. Let $\varphi\colon \cE \to \cF$ be such that $j^*\varphi = 0$. Since $\cE$ satisfies (BMP2) (we only use the surjectivity), $\varphi^E = 0$ for any $E \in \cV_{\delta Z}$ as well. Since $i^{[*]}\varepsilon$ is the identity map, we may lift $\varphi$ to $\psi\colon i_*i^{[*]}\cE \to \cF$ by setting $i^{[*]}\psi = i^{[*]}\varphi$. Thus $\ker(\gamma) \subset \im(\alpha)$.
 
 It remains to show that $\gamma$ is surjective. Since $X$ is finite, we may assume by factoring $j$ that $Z$ is a closed stratum $X_s$. For any edge $E\colon s\text{---}t$ (so $s \le t$ and $t \in U$), by (BMP2) for $\cE$, $\varphi^t$ induces a map $\cE^E \to \cF^E$, which we define to be $\varphi^E$:
 \[
  \xymatrix{
   \cE^t \ar[r]^{\varphi^t} \ar[d]_{\rho^\cE_{t, E}} & \cF^t \ar[d]^{\rho^\cF_{t, E}} \\
   \cE^E \ar@{-->}[r]^{\varphi^E} & \cF^E
  }
 \]
 Consider the following diagram:
 \[
  \xymatrix{
   \Gamma(\{ >s \}, \cE) \ar[r] \ar[d]_{u^\cE_s} & \Gamma(\{ >s \}, \cF) \ar[d]^{u^\cF_s} \\
   \bigoplus_{E \in \cV_{\delta s}}\cE^E \ar[r]^{\bigoplus\varphi^E} & \bigoplus_{E \in \cV_{\delta s}}\cF^E \\
   \cE^x \ar@{-->}[r]^{\varphi^s} \ar[u]^{d^\cE_s} & \cF^s \ar[u]_{d^\cF_s}
  }
 \]
 By (BMP3') and (BMP4') for both $\cE$ and $\cF$, we have $\im(\bigoplus \varphi^E \circ d^\cE_s) \subset \im(d^\cF_s)$. Since $\cE^x$ is projective by (BMP1), there exists $\varphi^s$ making the bottom square commute.
\end{proof}
\begin{rmk} \label{rmk:naive-functors-Fiebig}
 One can similarly define functors $h^{[*]}$ and $h^{[!]}$ for any inclusion $h\colon Y \hookrightarrow X$. It is clear that $j^{[*]} = j^{[!]}$ for an open inclusion $j$ and that all functors satisfy the usual composition property. For the inclusion $i_s\colon X_s \hookrightarrow X$ of a stratum, $i^{[!]}_s$ recovers Fiebig's notion of \emph{partial costalk}.
 
 Fiebig defines an exact structure on a full subcategory of moment graph sheaves that depends on the partial order on the strata~\cite[\S4.1]{Fie-verma}. The ``canonical short exact sequence'' of~\cite[Lemma~4.3]{Fie-verma} is the analogue of our ``distinguished triangle $i_*i^{[!]} \to \id \to j_*j^* \to $.''
\end{rmk}
\begin{rmk} \label{rmk:naive-functors-fans}
 For sheaves on fans, the analogue of Proposition~\ref{prop:BMP-pseudo-dt} (when $Z$ is a closed stratum) is stated in~\cite[Lemma~6.6.4]{BL}.
\end{rmk}

For $s \in \cV(X)$, let $i_s\colon X_s \hookrightarrow X$ be the inclusion of a stratum. Then the functor $i_s^{[*]}$ clearly restricts to $\BMPe(X) \to \BMPe(X_s)$. The analogous statement for $i_s^{[!]}$ is not true in general; a partial costalk of a BMP sheaf may not be graded free. We isolate this property as an additional condition on $X$:
\begin{enumerate}
 \item[\bf(V)] For all $s \in \cV(X)$, $i_s^{[!]}$ restricts to a functor $\BMPe(X) \to \BMPe(X_s)$.
\end{enumerate}
\begin{rmk} \label{rmk:bmp-verma-flag}
 In Fiebig's language, this says that that BMP sheaves on $X$ admit a Verma flag. This is known in the following cases: moment graphs associated to projective varieties with a suitable torus action~\cite{BL}; regular and subregular Bruhat graphs~\cite{Fie-coxeter} (see~Proposition~\ref{prop:bruhat-V}); (not necessarily rational) fans in the toric setting~\cite[Theorem~6.6.1(3)]{BL}.
\end{rmk}
\begin{lem} \label{lem:V-open}
 Let $X$ be a moment graph satisfying condition {\bf (V)}. Let $U$ be an open union of strata, and let $i_s^U\colon X_s \hookrightarrow U$ be the inclusion of a stratum that is closed in $U$. Then $(i_s^U)^{[!]}$ restricts to a functor $\BMPe(U) \to \BMPe(X_s)$.
\end{lem}
\begin{proof}
 Let $i_s\colon X_s \hookrightarrow U$ and $j\colon U \hookrightarrow X$ be the inclusions. Given $\cE \in \BMPe(U)$, since $U$ is open, there exists $\cF \in \BMPe(X)$ with $j^{[!]}\cF = j^*\cF = \cE$. Then $(i_s^U)^{[!]}\cE = (i_s^U)^{[!]}j^{[!]}\cF = i_s^{[!]}\cF$, which lies in $\BMPe(X_s)$ by {\bf (V)}.
\end{proof}

\subsection{The mixed derived category}
\label{ss:recollement}

\begin{defn}
 Define the \emph{(equivariant) mixed derived category} of $X$ by
 \[
  \Dmixe(X) := \Kb\BMPe(X).
 \]
\end{defn}
This category has a grading shift $\{1\}$ inherited from $\BMPe(X)$ and a new cohomological shift $[1]$. Define the \emph{Tate twist}
\[
 \la 1 \ra := \{-1\}[1].
\]

Let $X$ be a finite moment graph over $\bk$. Let $j\colon U \hookrightarrow X$ be an inclusion of an open union of strata, with closed complement $i\colon Z \hookrightarrow X$. The naive sheaf functors $i_*$ and $j^*$ clearly preserve BMP sheaves, hence induce exact functors denoted similarly between the mixed derived categories.
\begin{prop}[cf.~\cite{AR-II}, Proposition~2.3] \label{prop:recollement}
 Assume that $X$ satisfies condition {\bf (V)} or $\bk$ is a regular local ring (or both). Then the functors $i_*\colon \Dmixe(Z) \to \Dmixe(X)$ and $j^*\colon \Dmixe(X) \to \Dmixe(U)$ each admit left and right adjoints, giving the recollement diagram
 \[
  \xymatrix@C=1.5cm{
   \Dmixe(Z) \ar[r]|{i_*} &
   \Dmixe(X) \ar[r]|{j^*} \ar@/_1pc/[l]_{i^{*}} \ar@/^1pc/[l]^{i^!} &
   \Dmixe(U). \ar@/_1pc/[l]_{j_!} \ar@/^1pc/[l]^{j_*}
  }
 \]
\end{prop}
We start with Lemmas~\ref{lem:recollement-key-j_shriek} and~\ref{lem:recollement-key-j_star} below, which isolate the only modification to the proof in~\cite{AR-II} in the key case when $Z$ is a closed stratum. Then we sketch the rest of the proof, following~\cite{AR-II}.

So assume that $Z = X_s$ is a closed stratum. First consider $j_!$ and $i^*$. For any $\cE \in \BMPe(X)$, denote by $\cE^+ \in \Dmixe(X)$ the image of the complex
\[
  \cE \xrightarrow{\varepsilon}  i_*i^{[*]}\cE
\]
in degrees 0 and 1. Let $\sD^+$ be the full triangulated subcategory of $\Dmixe(X)$ generated by $\cE^+$ for all $\cE \in \BMPe(X)$, and let $\iota \colon \sD^+ \to \Dmixe(X)$ be the inclusion.
\begin{lem} \label{lem:recollement-key-j_shriek}
 \begin{enumerate}
  \item For any $\cF \in \sD^+$ and $\cG \in \Dmixe(X)$, the morphism
  \[
   \Hom_{\Dmixe(X)}(\iota\cF, \cG) \to \Hom_{\Dmixe(U)}(j^*\iota\cF, j^*\cG)
  \]
  induced by $j^*$ is an isomorphism.
  \item The composition $j^* \circ \iota\colon \sD^+ \to \Dmixe(U)$ is essentially surjective.
 \end{enumerate}
\end{lem}
\begin{proof}
 \begin{enumerate}
  \item By a standard five-lemma argument, this reduces to showing that
  \[
   \Hom_{\Dmixe(X)}(\iota\cE^+, \cG[n]) \to \Hom_{\Dmixe(U)}(j^*\iota\cE^+, j^*\cG[n])
  \]
  is an isomorphism for all $\cE, \cG \in \BMPe(X)$ and $n \in \Z$. If $n \notin \{-1, 0\}$, both sides are zero. The case $n = -1$ and $n = 0$ follow from Proposition~\ref{prop:BMP-pseudo-dt}. Indeed, for $n = -1$, the right hand side is zero because $j^*i_*i^{[*]}\cE = 0$, and the left hand side is 0 by the injectivity of $\alpha$ in Proposition~\ref{prop:BMP-pseudo-dt}. For $n = 0$, $\ker\gamma \subset \im\alpha$ gives the injectivity, while $\gamma$ surjective implies surjectivity.
  
  \item For any $\cF \in \BMPe(U)$, there exists $\cE \in \BMPe(X)$ such that $j^*\cE = \cF$. Now observe that $j^*\cE^+ = j^*\cE$. \qedhere
 \end{enumerate}
\end{proof}

Still assuming that $Z = X_s$, now consider $j_*$ and $i^!$. If $X$ satisfies {\bf (V)}, then for any $\cE \in \BMPe(X)$, we let $\cE^- \in \Dmixe(X)$ be the image of the complex
\[
  i_*i^{[!]}\cE \xrightarrow{\eta} \cE
\]
in degrees $-1$ and $0$, and Lemma~\ref{lem:recollement-key-j_shriek} dualizes. So suppose not. In this case, $i_*i^{[!]}\cE$ may no longer be BMP, so we need to modify $\cE^-$ using a resolution to obtain an object in $\Dmixe(X)$. First, since $\BMPe(X)$ is a full subcategory of $\She(X)$, we may view $\Dmixe(X)$ as a full subcategory of $\Kb\She(X)$. For any $\cE \in \BMPe(X)$ and any finite resolution
\[
  P_\bullet\colon 0 \to P_m \xrightarrow{f_m} \cdots \xrightarrow{f_3} P_2 \xrightarrow{f_2} P_1 \xrightarrow{f_1} i^{[!]}\cE \to 0
\]
of $i^{[!]}\cE \in \She(X_s)$ by BMP sheaves (i.e.~graded free $R$-modules of finite rank), let $\cE^-_{P_\bullet} \in \Dmixe(X)$ be the image of the complex
\[
 i_*P_m \xrightarrow{i_*f_m} \cdots \xrightarrow{i_*f_3} i_*P_2 \xrightarrow{i_*f_2} i_*P_1 \xrightarrow{\eta \circ (i_*f_1)} \cE,
\]
where $\cE$ is in degree 0. Let $\sD^-$ be the full triangulated subcategory of $\Dmixe(X)$ generated by all such $\cE^-_{P_\bullet}$, and let $\iota\colon \sD^- \to \Dmixe(X)$ be inclusion. The following result is dual to Lemma~\ref{lem:recollement-key-j_shriek}.
\begin{lem} \label{lem:recollement-key-j_star}
 \begin{enumerate}
  \item For any $\cF \in \Dmixe(X)$ and $\cG \in \sD^-$, the morphism
  \[
   \Hom_{\Dmixe(X)}(\cF, \iota \cG) \to \Hom_{\Dmixe(U)}(j^*\cF, j^*\iota\cG)
  \]
  induced by $j^*$ is an isomorphism.
  \item The composition $j^* \circ \iota\colon \sD^- \to \Dmixe(U)$ is essentially surjective.
 \end{enumerate}
\end{lem}
\begin{proof}[Proof of Lemma~\ref{lem:recollement-key-j_star}]
 \begin{enumerate}
  \item Again by the five lemma, we reduce to showing that
  \[
   \Hom_{\Dmixe(X)}(\cF[n], \iota\cE^-_{P_\bullet}) \to \Hom_{\Dmixe(U)}(j^*\cF[n], j^*\iota\cE^-_{P_\bullet})
  \]
  is an isomorphism for all $\cF \in \BMPe(X)$, $\cE^-_{P_\bullet}$, and $n \in \Z$. Since $j^*\iota \cE^-_{P_\bullet} \cong j^*\cE$ lives in degree 0, the right hand side is zero for $n \neq 0$, and the left hand side is also zero since $i^{[*]} \dashv i_*$ and $i^{[*]}\cF$ is projective. For $n = 0$, use Proposition~\ref{prop:BMP-pseudo-dt} as before.
  \item Let $\cE \in \BMPe(X)$. Since $\bk$ is a regular local ring, $i^{[!]}\cE$ has some finite free resolution $P_\bullet$. Then $j^*\cE^-_{P_\bullet} = j^*\cE$. \qedhere
 \end{enumerate}
\end{proof}

We sketch the rest of the proof, which is now exactly as in~\cite{AR-II}.
\begin{proof}[Proof of Proposition~\ref{prop:recollement}]
 We construct $j_!$ and $i^*$. First, in the case $Z = X_s$, Lemma~\ref{lem:recollement-key-j_shriek} implies that $j^* \circ \iota$ is an equivalence. Define
 \[
  j_! := \iota \circ (j^* \circ \iota)^{-1} \colon \Dmixe(U) \to \Dmixe(X).
 \]
 Then $j^*j_! \cong \id$. Lemma~\ref{lem:recollement-key-j_shriek} also implies that
 \[
  \Hom_{\Dmixe(X)}(j_!\cF, \cG) \to \Hom_{\Dmixe(U)}(j^*j_!\cF, j^*\cG) \cong \Hom_{\Dmixe(U)}(\cF, j^*\cG)
 \]
 is an isomorphism for all $\cF, \cG \in \Dmixe(X)$, so $j_! \dashv j^*$. Now, in the setting of Proposition~\ref{prop:recollement}, define $j_!$ by induction on the number of strata in $Z$. It can be shown that, for any $\cF \in \Dmixe(X)$, there exist unique objects $\cF' \in \Dmixe(U)$ and $\cF'' \in \Dmixe(Z)$ and a unique distinguished triangle
 \[
  j_!\cF' \to \cF \to i_*\cF'' \to.
 \]
 Then necessarily $j^*\cF \cong \cF'$, and $i^*$ can be uniquely defined up to isomorphism by setting $i^*\cF = \cF''$.
 
 The argument for $j_*$ and $i^!$ is entirely dual; when defining $j_*$ inductively,\\Lemma~\ref{lem:V-open} ensures that we can invoke {\bf (V)} at each step, for adding one closed stratum.
\end{proof}
\begin{rmk} \label{rmk:restriction-two-strata}
 If the naive restriction $i^{[*]}\colon \She(X) \to \She(Z)$ preserves BMP sheaves, it induces a functor $i^{[*]}\colon \Dmixe(X) \to \Dmixe(Z)$ that is left adjoint to $i_*$, so $i^* \cong i^{[*]}$. The same is true for $i^{[!]}$; in particular, if $X$ satisfies {\bf (V)}, then $i_s^! \cong i_s^{[!]}$ for all $s \in \cV(X)$.
 
 The statement for $i^{[*]}$ applies for example when every connected component of the underlying graph of $Z$ consists of one or two strata. Indeed, if $\cF \in \BMPe(X)$, then $i^{[*]}\cF$ clearly satisfies (BMP1), (BMP2), and (BMP4). We may check (BMP3) on each connected component: there is nothing to check if it has one stratum, while for a component $\{s, t\}$ with $s \le t$, (BMP3) is just the surjectivity of $\Gamma(\{s, t\}, \cF) \to \cF^t$, which follows from (BMP4).
\end{rmk}
Achar--Riche next define these functors for locally closed inclusions by composition \cite[\S2.5]{AR-II}; their results transfer to our setting with the same proof (but using the duality in recollement rather than Verdier duality). Thus we have the full assortment of $*$- and $!$-type pushforward and pullback functors between locally closed unions of ``strata'' of an arbitrary moment graph, satisfying the usual composition and adjunction properties.

\subsection{Proper smooth stratified morphisms}
\label{ss:pss}

Let $X, X'$ be moment graphs over the same free $\bk$-module $\fh$.
\begin{defn}
 A \emph{morphism of moment graphs} $f\colon X \to X'$ is a set-map
 \[
  f\colon \cV(X) \to \cV(X')
 \]
 such that for any edge $E\colon s\text{---}t$ in $X$, either $f(s) = f(t)$, or $f(s)$ and $f(t)$ are connected by an edge labeled $\alpha_X(E)$.
\end{defn}
For a stratum $X'_s$ ($s \in \cV(X')$), denote by $f^{-1}(X'_s)$ the sub-moment graph of $X$ corresponding to $f^{-1}(s) \subset \cV(X)$.
\begin{defn} \label{def:pullback}
 Given a morphism $f\colon X \to X'$, we define the \emph{pullback functor}
 \[
  f^*\colon \She(X') \to \She(X).
 \]
 Given $\cF \in \She(X')$, $f^*\cF \in \She(X)$ is defined as follows. For $s \in \cV(X)$, set
 \[
  (f^*\cF)^s = \cF^{f(s)}.
 \]
 Let $E\colon s\text{---}t$ be an edge in $X$. If $f(s) = f(t)$, then set
 \[
  (f^*\cF)^E = \cF^{f(s)}/\alpha(E) = \cF^{f(t)}/\alpha(E),
 \]
 and let $\rho^{f^*(\cF)}_{s, E}$, $\rho^{f^*(\cF)}_{t, E}$ be the quotient map. Otherwise, there is an edge\\$E'\colon f(s)\text{---}f(t)$, and we set
 \[
  (f^*\cF)^E = \cF^{E'}, \qquad \rho^{f^*\cF}_{s, E} = \rho^\cF_{f(s), E'}, \qquad \rho^{f^*\cF}_{t, E} = \rho^\cF_{f(t), E'}.
 \]
\end{defn}
\begin{defn} \label{defn:pss}
 Let $d \in \Z$. A morphism $f\colon X \to X'$ is called \emph{PSS-$d$} if it satisfies the following conditions:
 \begin{enumerate}
  \item $f^*\colon \She(X') \to \She(X)$ restricts to $f^*\colon \BMPe(X') \to \BMPe(X)$;
  \item the induced exact functor $f^*\colon \Dmixe(X') \to \Dmixe(X)$ admits an exact right adjoint $f_*\colon \Dmixe(X) \to \Dmixe(X')$, and $f^*$ is left adjoint to $f^*\{2d\}$. (That is, $f^*$ admits a biadjoint up to shift.)
 \end{enumerate}
 If $f$ is PSS-$d$, we set $f_! = f_*$ and $f^! = f^*\{2d\}$.
\end{defn}
This is meant to be an analogue of proper smooth stratified (in the sense of~\cite[\S2.6]{AR-II}) of relative dimension $d$. Note that we do not define analogues of proper, smooth, or stratified morphisms individually.

\begin{prop}[cf.~\cite{AR-II}, Proposition~2.8] \label{prop:pss}
 Let $f\colon X \to X'$ be PSS-$d$. Let $h'\colon Y' \hookrightarrow X'$ be a locally closed inclusion. Consider the following diagram:
 \[
  \xymatrix{
   f^{-1}(Y') \ar@{^(->}[r]^-h \ar[d]_f & X \ar[d]^f \\
   Y' \ar@{^(->}[r]_{h'} & X'.
  }
 \]
 Then
 \begin{enumerate}
  \item $f\colon f^{-1}(Y') \to Y'$ is PSS-$d$;
  \item we have the following natural isomorphisms of functors:
  \[
   \begin{array}{ccccccc}
    f_*h_* &\cong& h_*f_*,& & f_*h_! &\cong& h_!f_*, \\
    h^*f^* &\cong& f^*h^*,& & h^!f^* &\cong& f^*h^*, \\
    f^*h_* &\cong& h_*f^*,& & f^*h_! &\cong& h_!f^*, \\
    h^*f_* &\cong& f_*h^*,& & h^!f_* &\cong& f_*h^!.
   \end{array}
  \]
 \end{enumerate}
\end{prop}
\begin{rmk}
 In fact, if $f^*\colon \BMPe(X') \to \BMPe(X)$ admits a biadjoint up to shift before passing to the homotopy category, then it is easy to see that $f^*\colon \BMPe(Y') \to \BMPe(f^{-1}(Y'))$ does also.
\end{rmk}
We need to argue differently from~\cite{AR-II} since we do not have a proper base change theorem at the non-mixed level. We need two preliminary lemmas.

Let $f\colon Y \to Y'$ be PSS-d. Consider the diagram
\begin{align*}
 \xymatrix@C=1.5cm{
  Z \ar@{^(->}[r]^i \ar[d]_f & Y \ar[d]_f & U \ar@{_(->}[l]_j \ar[d]_f \\
  Z' \ar@{^(->}[r]_i & Y' & U' \ar@{_(->}[l]^j
 }
\end{align*}
where $i\colon Z' \hookrightarrow Y'$ is a closed inclusion with open complement $j\colon U' \hookrightarrow Y'$, and $Z$ and $U$ are full preimages under $f$. Then we have the diagram
\begin{equation} \label{diag:pss-recollement}
 \begin{array}{c}\xymatrix@C=1.5cm{
  \Dmixe(Z) \ar[r] & \Dmixe(Y) \ar[r] \ar@<-0.6ex>[d]_{f_*} \ar@<0.6ex>[l] \ar@<-0.6ex>[l] & \Dmixe(U) \ar@<0.6ex>[l] \ar@<-0.6ex>[l] \\
  \Dmixe(Z') \ar[r] \ar@<-0.6ex>[u]_{f^*} & \Dmixe(Y') \ar[r] \ar@<-0.6ex>[u]_{f^*} \ar@<0.6ex>[l] \ar@<-0.6ex>[l] & \Dmixe(U') \ar@<-0.6ex>[u]_{f^*} \ar@<0.6ex>[l] \ar@<-0.6ex>[l]
 }\end{array}
\end{equation}
where the rows are recollement diagrams. To show (1), it suffices in this set-up to construct a biadjoint $f_*$ up to shift for the left and right hand columns. We will use the following lemmas.
\begin{lem} \label{lem:commute-j-shriek}
 In this set-up, $f^*j_! \cong j_!f^*$.
\end{lem}
\begin{proof}
 By factoring $j$, we reduce to the case that $Z'$ is a closed stratum $X_s$. It suffices to show that for any $\cF \in \Dmixe(U')$ and $\cG \in \Dmixe(X)$, the map
 \begin{align} \label{eqn:commute-j-shriek}
  \Hom_{\Dmixe(X)}(f^*j_!\cF, \cG) \to \Hom_{\Dmixe(U)}(j^*f^*j_!\cF, j^*\cG)
 \end{align}
 induced by $j^*$ is an isomorphism. Indeed, we would then get an isomorphism
 \begin{multline*}
  \Hom_{\Dmixe(X)}(f^*j_!\cF, \cG) \cong \Hom_{\Dmixe(U)}(j^*f^*j_!\cF, j^*\cG) \\
  \cong \Hom_{\Dmixe(U)}(f^*j^*j_!\cF, j^*\cG) \cong \Hom_{\Dmixe(U)}(f^*\cF, j^*\cG) \\
  \cong \Hom_{\Dmixe(X)}(j_!f^*\cF, \cG)
 \end{multline*}
 functorial in $\cF$ and $\cG$, and the result follows by a Yoneda-type argument.
 
 The proof that \eqref{eqn:commute-j-shriek} is an isomorphism is similar to the proof of Lemma~\ref{lem:recollement-key-j_shriek}. First, recall from Lemma~\ref{lem:recollement-key-j_shriek} and the proof of Proposition~\ref{prop:recollement} that the essential image of $j_!: \Dmixe(U') \to \Dmixe(X')$ is the full triangulated subcategory of $\Dmixe(X')$ generated by $\cE^+$ for all $\cE \in \BMPe(X')$. By the five lemma, it therefore suffices to show that
 \[
  \Hom_{\Dmixe(X)}(f^*\cE^+, \cG[n]) \to \Hom_{\Dmixe(U)}(j^*f^*\cE^+, j^*\cG[n])
 \]
 is an isomorphism for all $\cE \in \BMPe(X')$, $\cG \in \BMPe(X)$, $n \in \Z$. For this, note that $f^*i_*i^{[*]}\cE \cong i_*i^{[*]}f^*\cE$, and in fact $f^*\cE^+ \cong (f^*\cE)^+$. Now apply the first short exact sequence of Proposition~\ref{prop:BMP-pseudo-dt} to the BMP sheaves $f^*\cE$ and $\cG$.
\end{proof}
The next lemma is part of~\cite[Theorem~2.5]{PS}.
\begin{lem} \label{lem:morphism-of-recollement}
 Consider the diagram
 \[
  \xymatrix@C=1.5cm{
   \cD_Z \ar[r] \ar[d]_{F_Z} & \cD \ar[r] \ar[d]_F \ar@<0.6ex>[l] \ar@<-0.6ex>[l] & \cD_U \ar[d]_{F_U} \ar@<0.6ex>[l] \ar@<-0.6ex>[l] \\
   \cE_Z \ar[r] & \cE \ar[r] \ar@<0.6ex>[l] \ar@<-0.6ex>[l] & \cE_U \ar@<0.6ex>[l] \ar@<-0.6ex>[l]
  }
 \]
 of exact functors between triangulated categories, where the rows are recollement diagrams. (To simplify notation, we denote all of $F$, $F_Z$, $F_U$ by $F$, and write $i_*$, $j^*$, etc.~as usual for the functors in either recollement.) Assume that $Fi_* \cong i_*F$ and $Fj^* \cong j^*F$. Then $Fj_! \cong j_!F$ $\Leftrightarrow$ $Fi^* \cong i^*F$ and $Fj_* \cong j_*F$ $\Leftrightarrow$ $Fi^! \cong i^!F$.
\end{lem}

\begin{rmk}
 If $f$ has the additional property that the preimage of every stratum has at most two strata, then Lemma~\ref{lem:commute-j-shriek} has the following simpler proof. Again assume that $Z'$ is a closed stratum. By Remark~\ref{rmk:restriction-two-strata}, $i^*$ in either row is the naive restriction, so $f^*i^* \cong i^*f^*$. Then Lemma~\ref{lem:morphism-of-recollement} applied with $F = f^*$ implies $j_!f^* \cong f^*j_!$. This applies for example to the morphism $\pi^s\colon \cB \to \cP^s$ defined in~\S\ref{ss:bmp-sbim}.
\end{rmk}

\begin{proof}[Proof of Proposition~\ref{prop:pss}]
 \begin{enumerate}
  \item Assume the set-up of \eqref{diag:pss-recollement}. We will construct a biadjoint $f_*$ up to shift for the left and right hand columns. First consider the left hand column. Since $i_*$ is a full embedding, it is enough to show that $f_*\colon \Dmixe(X) \to \Dmixe(X')$ restricts to $\Dmixe(Z) \to \Dmixe(Z')$. So let $\cF \in \Dmixe(Z)$. To show that $f_*i_*\cF$ is supported on $Z'$, by recollement it suffices to show that $j^*f_*i_*\cF = 0$. This is true since, for any $\cG \in \Dmixe(U')$,
  \begin{multline*}
   \Hom_{\Dmixe(U')}(\cG, j^*f_*i_* \cF) \cong \Hom_{\Dmixe(Y)}(f^*j_!\cG, i_*\cF) \\
   \overset{\text{Lemma~}\ref{lem:commute-j-shriek}}{\cong} \Hom_{\Dmixe(Y)}(j_!f^*\cG, i_*\cF)
   \cong \Hom_{\Dmixe(U)}(f^*\cG, j^*i_*\cF) = 0.
  \end{multline*}
  
  Now consider the right hand column. Recall that in recollement, $j^*$ identifies $\Dmixe(U)$ as a quotient of $\Dmixe(Y)$ by $\Dmixe(Z)$ in an appropriate sense, and similarly for the bottom row. The previous paragraph therefore implies that $f_*\colon \Dmixe(Y) \to \Dmixe(Y')$ induces an exact functor $f_*\colon \Dmixe(U) \to \Dmixe(U')$ satisfying $j^*f_* \cong f_*j^*$.

  It remains to show the biadjunction up to shift for $U \to U'$. Since $j_!$ is fully faithful, this would follow from the same for $Y \to Y'$ if we knew that $j_!f^* \cong f^*j_!$ and $j_!f_* \cong f_*j_!$. The first isomorphism is Lemma~\ref{lem:commute-j-shriek}. Clearly $i_*f^* \cong i_*f^*$, and by adjunction for $Y \to Y'$ and $Z \to Z'$, we get $i^*f_* \cong f_*i^*$. This implies $j_!f_*\cong f_*j_!$ by Lemma~\ref{lem:morphism-of-recollement} applied with $F = f_*$.

  \item In the same set-up, we now have the diagram
  \[
   \xymatrix@C=1.5cm{
    \Dmixe(Z) \ar[r] \ar@<-0.6ex>[d]_{f_*} & \Dmixe(Y) \ar[r] \ar@<-0.6ex>[d]_{f_*} \ar@<0.6ex>[l] \ar@<-0.6ex>[l] & \Dmixe(U) \ar@<0.6ex>[l] \ar@<-0.6ex>[l] \ar@<-0.6ex>[d]_{f_*} \\
    \Dmixe(Z') \ar[r] \ar@<-0.6ex>[u]_{f^*} & \Dmixe(Y') \ar[r] \ar@<-0.6ex>[u]_{f^*} \ar@<0.6ex>[l] \ar@<-0.6ex>[l] & \Dmixe(U') \ar@<-0.6ex>[u]_{f^*} \ar@<0.6ex>[l] \ar@<-0.6ex>[l] 
   }
  \]
  where $f^*$ commutes up to natural isomorphism with $i_*$, $j^*$, and $j_!$, and $f^* \dashv f_* \dashv f^*\{2d\}$ in each column. (Note, however, that we no longer know any commutativity involving $f_*$.) To prove (2), it suffices to show that all squares here commute up to natural isomorphism. These are related as follows:
  \[
   \xymatrix@C=-0.7cm@R=0.5cm{
                                              & i_*f^* \cong i_*f^* \ar@{}[dl]|{\dashvdl} \ar@{}[dr]|{\dashvdr} &                                           \\
    f_*i^* \cong i^*f_*                       &                                                                 & f_*i^! \cong i^!f_*                       \\
    f_*j_! \cong j_!f_* \ar@{}[dr]|{\dashvdr} &                                                                 & f_*j_* \cong j_*f_* \ar@{}[dl]|{\dashvdl} \\
                                              & j^*f^* \cong f^*j^*                                             &
   }
   \hskip 1.5cm
   \xymatrix@C=-0.7cm@R=0.5cm{
                                              & i_*f_* \cong i_*f_* \ar@{}[dl]|{\dashvdl} \ar@{}[dr]|{\dashvdr} &                                           \\
    f^*i^* \cong i^*f^* \ar@{<=>}[d]          &                                                                 & f^*i^! \cong i^!f^* \ar@{<=>}[d]          \\
    f^*j_! \cong j_!f^* \ar@{}[dr]|{\dashvdr} &                                                                 & f^*j_* \cong j_*f^* \ar@{}[dl]|{\dashvdl} \\
                                              & j^*f_* \cong f_*j^*                                             &
   }
  \]
  Here, the two equivalences follow from Lemma~\ref{lem:morphism-of-recollement}, which however does not apply to $F = f_*$. In any case, the desired isomorphisms follow from the following ones: $i_*f^* \cong i_*f^*$, $j^*f^* \cong f^*j^*$, and $f^*j_! \cong j_!f^*$ (Lemma~\ref{lem:commute-j-shriek}). \qedhere
 \end{enumerate}
\end{proof}

\section{Mixed perverse sheaves}
\label{s:mixed-perverse-sheaves}

In this section, we assume for simplicity that $\bk$ is a field.

\subsection{Perverse t-structure}
\label{ss:perverse-t-structure}

Let $X$ be a moment graph over $\bk$. A BMP sheaf on a single stratum $X_s$ is simply a finite rank graded free $R$-module. Denote by $\ubk_{X_s} \in \BMPe(X_s)$ the regular module $R$. We call it the \emph{constant sheaf} on $X_s$.

From now on, we assume that $X$ is equipped with a \emph{dimension function} $s \mapsto d_s\colon \cV(X) \to \Z$. For each $s \in \cV(X)$, define the \emph{standard} and \emph{costandard} sheaves
\[
 \Delta_s := i_{s!}\ubk_{X_s}\{ d_s \}, \qquad \nabla_s := i_{s*}\ubk_{X_s}\{ d_s \}
\]
in $\Dmixe(X)$.

For a single stratum $X_s$, the perverse t-structure on $\Dmixe(X_s)$ will be analogous to the one in~\cite[\S6.7]{BL}. A BMP sheaf on $X_s$ is a finite rank graded free $R$-module. Let $\p D^{\le0}$ (resp.~$\p D^{\ge0}$) be the full subcategory of $\Dmixe(X_s)$ consisting of complexes that are homotopy equivalent to a complex $M$, where $M^i \cong \bigoplus_k R\{n_k\}$ with $n_k \ge i + d_s$ (resp.~$n_k \le i + d_s$).
\begin{lem}
 The pair $(\p D^{\le0}, \p D^{\ge0})$ defines a t-structure on $\Dmixe(X_s)$.
\end{lem}
\begin{proof}
 Clearly $\p D^{\le0}[1] \subset \p D^{\le0}$ and $\p D^{\ge0} \subset \p D^{\ge0}[1]$. If $M \in \p D^{\le0}$ and $N \in \p D^{\ge1}$, then $\Hom_{R\lh\gmod}(M^i, N^i) = 0$, so $\Hom_{\Dmixe(X_s)}(M, N) = 0$. Given $M \in \Dmixe(X_s)$, we must find a distinguished triangle $L \to M \to N \to$ with $L \in \p D^{\le0}$ and $N \in \p D^{\ge1}$. Each $M^i$ has a canonical decomposition $M^i = \bigoplus_j M^i_j$, where $M^i_j$ is isomorphic to a direct sum of copies of $R\{j\}$. Decompose $d^i\colon M^i \to M^{i+1}$ accordingly as $d^i_{jk}\colon M^i_j \to M^{i+1}_{j+k}$ (this is 0 unless $k \ge 0$). Now, define $L$ by
 \[
  L^i = \ker(d^i_{i+d_s, 0}\colon M^i_{i+d_s} \to M^{i+1}_{i+d_s}) \oplus \left( \bigoplus_{j > i+d_s} M^i_j \right)
 \]
 and differentials induced from $M$, and define $N = M/L$. Since $\bk$ is a field, any nonzero map $R \to R$ in $R\lh\gmod$ is an isomorphism. So $d^i_{i+d_s, 0}$ has graded free kernel and cokernel, so that $L$ and $N$ are in $\Dmixe(X_s)$. Clearly $L \in \p D^{\le0}$, and it is also straightforward to see that $N \in \p D^{\ge1}$.
\end{proof}

The following result is the equivariant analogue of~\cite[Lemma~3.2]{AR-II} and is proved in the same way.
\begin{lem}\label{lem:vanishing-delta-nabla}
 Let $s, t \in \cV(X)$. Then we have
 \[
  \Hom_{\Dmixe(X)}^\bullet (\Delta_s, \nabla_t [i]) \cong
   \begin{cases}
    R & \text{if $s = t$ and $i = 0$;} \\
    0 & \text{otherwise,}
   \end{cases}
 \]
 as graded $R$-modules.
\end{lem}

\begin{defn}
 For a finite moment graph $X$, the \emph{perverse t-structure} on $\Dmixe(X)$ is defined by recollement, as in~\cite{BBD}:
 \begin{align*}
  \p\Dmixe(X)^{\le0} &= \{ \cF \in \Dmixe(X) \mid i_s^*\cF \in \p\Dmixe(X_s)^{\le 0} \text{ for all } s \in \cV(X) \}, \\
  \p\Dmixe(X)^{\ge0} &= \{ \cF \in \Dmixe(X) \mid i_s^!\cF \in \p\Dmixe(X_s)^{\ge 0} \text{ for all } s \in \cV(X) \}.
 \end{align*}
\end{defn}

\begin{prop}[cf.~\cite{AR-II}, Proposition~3.4]
 The perverse t-structure on\\$\Dmixe(X)$ is uniquely characterized by each of the following statements:
 \begin{enumerate}
  \item $\p\Dmixe(X)^{\le0}$ is generated under extensions by the $\Delta_s\la n \ra[m]$ with $s \in \cV(X)$, $n \in \Z$, and $m \ge 0$.
  \item $\p\Dmixe(X)^{\ge0}$ is generated under extensions by the $\nabla_s\la n \ra[m]$ with $s \in \cV(X)$, $n \in \Z$, and $m \le 0$.
 \end{enumerate}
\end{prop}
Let $\Pmixe(X)$ denote the heart. Its objects are called \emph{mixed perverse sheaves}. Tate twist is t-exact for the perverse t-structure, hence restricts to an exact endofunctor of $\Pmixe(X)$. For each $s \in \cV(X)$, define
\[
 \IC_s := \im(\p\tau_{\ge0}\Delta_s \to \p\tau_{\le0}\nabla_s) \in \Pmixe(X),
\]
where $\p\tau_{\le i}$, $\p\tau_{\ge i}$ are the truncation functors. It follows from the general recollement theory of~\cite{BBD} that $\Pmixe(X)$ is a finite length abelian category with simples $\{ IC_s \mid s \in \cV(X) \}$ up to isomorphism and Tate twist.

The next lemma addresses the behavior of (co)standard sheaves under PSS morphisms. For us, the following partial analogue of~\cite[Lemma~3.8]{AR-II} will suffice.
\begin{lem}\label{lem:pss-delta-nabla}
 Let $f\colon X \to Y$ be PSS-$d$. Suppose that $X_s \subset f^{-1}(Y_t)$, i.e.~$f(s) = t$. Then
 \begin{align*}
  f_*\nabla_s &\cong \nabla_t \{ d_s - d_t \}, \\
  f_*\Delta_s &\cong
  \begin{cases}
   \Delta_t \{d_s - d_t\} &\mbox{ if } X_s \mbox{ is a closed stratum in } f^{-1}(Y_t); \\
   \Delta_t \{ d_s - d_t - 2d \} &\mbox{ if } X_s \mbox{ is an open stratum in } f^{-1}(Y_t).
  \end{cases}
 \end{align*}
\end{lem}
\begin{proof}
 Consider the following diagram:
 \[
  \xymatrix{
   X_s \ar@/^1pc/@{^(->}[rr]^-{i_s} \ar@{^(->}[r]_-{i'_s} & f^{-1}(Y_t) \ar[d]_f \ar[r]_-h & X \ar[d]^f \\
    & Y_t \ar@{^(->}[r]_-{i_t}  & Y
  }
 \]
 In what follows, use the identification $\Dmixe(X_s) \cong \Dmixe(Y_t)$ induced from the obvious identification $\BMPe(X_s) \cong \BMPe(Y_t)$ sending the constant sheaf to the constant sheaf.
 
 By definition, $f^*$ is the pullback functor, and by Remark~\ref{rmk:restriction-two-strata}, $i^{\prime *}_s$ is the naive restriction, so $i^{\prime *}_sf^* \cong \id$. By adjunction, $f_*i'_{s*} \cong \id$. Hence
 \begin{multline*}
  f_*\nabla_s = f_*i_{s*} \ubk_{X_s}\{ d_s \} \cong f_*h_*i'_{s*} \ubk_{X_s}\{ d_s \} \overset{\text{Prop.~\ref{prop:pss}}}{\cong} i_{t*}f_*i'_{s*} \ubk_{X_s}\{ d_s \} \\
  \cong i_{t*} \ubk_{Y_t}\{ d_s \} \cong \nabla_t\{d_s - d_t\}.
 \end{multline*}

 If $X_s$ is a closed stratum in $f^{-1}(Y_t)$, then $i^{\prime !}_s = i^{\prime *}_s$ is again the naive restriction, so $f_*i'_{s*} \cong \id$ and $f_*\Delta_s \cong \Delta_t\{d_s - d_t\}$ by dualizing the argument above.
 
 Suppose $X_s$ is an open stratum in $f^{-1}(Y_t)$. Then $f_!i'_{s!} \dashv i^{\prime !}_sf^! \cong i^{\prime *}_sf^*\{2d\} \cong \id\{2d\}$, so $f_!i'_{s!} \cong \id\{-2d\}$, and we again use Proposition~\ref{prop:pss} to show that $f_*\Delta_s \cong \Delta_t \{ d_s - d_t - 2d \}$.
\end{proof}

\subsection{Constructible mixed derived category and extension to infinite moment graphs}
\label{ss:constructible}
From now on, we assume that $X$ satisfies condition {\bf (V)} from the end of~\S\ref{ss:naive-functors}.

Denote by $\BMPc(X)$ the category obtained from $\BMPe(X)$ by applying $\bk \otimes_R (-)$ to graded Homs. That is, the objects of $\BMPc(X)$ are again BMP sheaves, and for two BMP sheaves $\cE, \cF$, we have
 \[
  \Hom_{\BMPc(X)}(\cE, \cF) = (\bk \otimes_R \Hom^\bullet_{\BMPe(X)}(\cE, \cF))_0,
 \]
i.e.~the degree 0 part of $\bk \otimes_R \Hom^\bullet_{\BMPe(X)}(\cE, \cF)$. This category has an induced grading shift $\{1\}$. The natural quotient functor
\[
 \For\colon \BMPe(X) \to \BMPc(X)
\] 
is called the \emph{forgetful functor}. Defining the graded Hom as before (now only a graded $\bk$-module), it follows from the definitions that the natural map
\begin{equation} \label{eq:equivariant-formality}
 \bk \otimes_R \Hom_{\BMPe(X)}^\bullet(\cE, \cF) \simto \Hom_{\BMPc(X)}^\bullet(\For\cE, \For\cF)
\end{equation}
is an isomorphism in $\bk\lh\gmod$.

\begin{rmk}
 The forgetful functor $\For\colon \Parity_T(X, \bk) \to \Parity_{(T)}(X, \bk)$ from $T$-equivariant to $T$-constructible parity sheaves satisfies the analogue of \eqref{eq:equivariant-formality}~\cite[Lemma~2.2(2)]{MR}. The equivalence $\Parity_T(X, \bk) \simto \BMPe(\bk \otimes_\Z \cG(X, T))$ therefore induces an equivalence $\Parity_{(T)}(X, \bk) \simto \BMPc(\bk \otimes_\Z \cG(X, T))$.
\end{rmk}

\begin{defn}
 Define the \emph{constructible mixed derived category} of $X$ by
 \[
  \Dmixc(X) := \Kb\BMPc(X).
 \]
\end{defn}
Let $\For\colon \Dmixe(X) \to \Dmixc(X)$ be the induced exact functor.

\begin{prop} \label{prop:BMP-graded-free-HOM}
 Let $X$ be a finite moment graph satisfying {\bf (V)}. For any $\cE, \cF \in \BMPe(X)$, $\Hom_{\BMPe(X)}^\bullet(\cE, \cF)$ is graded free over $R$.
\end{prop}
\begin{proof}
 Choose a total order $s_1 > \cdots > s_r$ on $\cV(X)$ refining the given partial order. For $k \ge 1$, consider the open inclusion $j_k\colon U_k = \{s_1, \ldots, s_k\} \hookrightarrow X$. We will induct on $k$ to show that $\Hom^\bullet_{\She(U_k)}(j_k^*\cE, j_k^*\cF)$ is graded free.
 
 This is clear for $k = 1$. For $k > 1$, let $j'_k\colon U_k \hookrightarrow U_{k+1}$ be the open inclusion with closed complement $i'_k\colon X_{s_k} \hookrightarrow U_{k+1}$, and let $i_k = j_{k+1} \circ i'_k\colon X_{s_k} \hookrightarrow X$. Applying Proposition~\ref{prop:BMP-pseudo-dt} to $j_{k+1}^*\cE$ and $j_{k+1}^*\cF\{n\}$ and summing over $n \in \Z$, we get a short exact sequence of graded $R$-modules
 \begin{multline*}
  0 \to \Hom_{\She(U_{k+1})}^\bullet(j_{k+1}^*\cE, i'_{k*}i_k^{[!]}\cF) \\
  \to \Hom_{\She(U_{k+1})}^\bullet(j_{k+1}^*\cE, j_{k+1}^*\cF) \to \Hom_{\She(U_k)}^\bullet(j_k^*\cE, j_k^*\cF) \to 0.
 \end{multline*}
 The first term is graded free by the condition {\bf (V)}, and the last term is graded free by induction. Hence so is the middle term.
\end{proof}

Proposition~\ref{prop:BMP-graded-free-HOM} implies that the constructible analogue of Proposition~\ref{prop:BMP-pseudo-dt} holds, and then that the whole theory of~\S\ref{s:recollement} generalizes to the constructible mixed derived category with the same proofs.

For a single stratum $X_s$, $\BMPc(X_s)$ naturally identifies with the category of finite dimensional $\bk$-vector spaces. Thus $\Kb\BMPc(X_s)$ has a natural t-structure, and shifting by the dimension function, we obtain as before the perverse t-structure on $\Dmixc(X)$. It is easy to check that the forgetful functor is perverse t-exact.

We have the following constructible analogue of Lemma~\ref{lem:vanishing-delta-nabla}:
\begin{lem}[cf.~\cite{AR-II}, Lemma~3.2] \label{lem:vanishing-delta-nabla-constructible}
 \[
  \Hom_{\Dmixc(X)}^\bullet (\Delta_s, \nabla_t [i]) \cong
   \begin{cases}
    \bk & \text{if $s = t$ and $i = 0$;} \\
    0   & \text{otherwise,}
   \end{cases}
 \]
 as graded $\bk$-modules.
\end{lem}
All other results of~\S\ref{ss:perverse-t-structure} continue to hold, with the same proof.

Finally, the entire theory readily extends to infinite moment graphs $X$ such that $\{ \le s\}$ is finite for all $s \in \cV(X)$, analogous to the extension to ind-varieties in~\cite[\S3.6]{AR-II}.

\subsection{Graded highest weight structure}
\label{ss:ghw-structure}

Let $X$ be a moment graph satisfying the condition {\bf (V)}, so that $\Pmixc(X)$ is defined. We now impose the following additional condition from~\cite[\S3.2]{AR-II}:
\begin{enumerate}
 \item[\bf(A2)] For each $s \in \cV(X)$, the objects $\Delta_s$ and $\nabla_s$ are perverse.
\end{enumerate}
\begin{prop}[cf.~\cite{AR-II}, Proposition~3.11] \label{prop:ghw-structure}
 If $X$ satisfies the condition~{\bf (A2)}, then $\Pmixc(X)$ is a graded highest weight category in the sense of~\cite[Definition~A.1]{AR-II}, and $\Delta_s\la n \ra$ (resp.~$\nabla_s\la n \ra$) are precisely the standard (resp.~costandard) objects.
\end{prop}
In this situation, we may speak of the full additive subcategory
\[
 \Tmixc(X) \subset \Pmixc(X)
\]
of tilting objects, whose objects are called \emph{mixed tilting sheaves}. For each $s \in \cV(X)$, there is a unique indecomposable tilting object $\cT_s$ that is supported on $\{ \le s\}$ and whose restriction to $X_s$ is $\ubk_{X_s}\{d_s\}$.
\begin{lem}[cf.~\cite{AR-II}, Lemma~3.15] \label{lem:tilt-perv-der-eq}
 The natural functors
 \[
  \Kb\Tmixc(X) \to \Db\Pmixc(X) \to \Dmixc(X)
 \]
 are equivalences of categories.
\end{lem}

\section{Bruhat moment graphs}
\label{s:bruhat}

In this section, we apply the general theory above to (partial) Bruhat graphs.

Throughout this section, fix a Coxeter system $(W, S)$ with $|S| < \infty$. Denote the identity element by $1$. Let $T \subset W$ denote the set of reflections (conjugates of simple reflections). For $I \subset S$, let $W_I \le W$ denote the subgroup generated by $I$. We say that $I$ is \emph{finitary} if $W_I$ is finite.

We also fix a field $\bk$ of characteristic not equal to 2, and a reflection faithful representation $\fh$ of $(W, S)$ over $\bk$, in the following sense.
\begin{defn}[\cite{Soe07}, Definition~1.5] \label{defn:reflection-faithful}
 A \emph{reflection faithful} representation $\fh$ of $(W, S)$ is a faithful, finite dimensional representation of $W$ such that for all $w \in W$, the fixed subspace $V^w$ has codimension 1 if and only if $w \in T$.
\end{defn}
We call the data $(W, \fh)$ a \emph{(reflection faithful) realization}.

\subsection{BMP sheaves and Soergel bimodules}
\label{ss:bmp-sbim}
Associated to $(W, \fh)$ will be two different classes of categories: moment-graph-theoretic and Soerge-theoretic. Pro\-position~\ref{prop:bmp-sbim} relates the two; this is a slight extension of Fiebig's results~\cite{Fie-verma, Fie-coxeter}.

We start with the moment graph side. For each $t \in T$, choose a nonzero form $\alpha_t \in V^*$ vanishing on $V^t$. Each $\alpha_t$ is unique up to a nonzero scalar, and everything that follows will be independent of this choice.

\begin{defn} (\cite[\S2.2]{Fie-verma})
 Let $I \subset S$ be finitary. The \emph{(partial) Bruhat (moment) graph} $\cP^I = \cP^I(W, \fh)$ is the following moment graph over $\fh$:
 \begin{itemize}
  \item vertices $W/W_I$, partially ordered by the induced Bruhat order;
  \item edge between $p, q \in W/W_I$ labeled $\alpha_t$, whenever $p = tq$ for some $t \in T$;
  \item dimension function
   \[
    p \mapsto d_p := \min \{\ell(w) \mid w \in p\}\colon W/W_I \to \Z.
   \]
 \end{itemize}
\end{defn}
In what follows, whenever we write $I$, we mean either $I = \emptyset$ or $I = \{s\}$ ($s \in S$). As a general rule, we will omit $\emptyset$ from the notation, and write $s$ instead of $\{s\}$. For example, $W_s = \{1, s\}$ and $\cP^s = \cP^{\{s\}}$. We write $\cB$ for $\cP^\emptyset$.

Much of the discussion below should generalize to any finitary $I$. We hope to return to this in a future work.

Let $s \in S$. The natural quotient map $W \to W/W_s$ defines a map of moment graphs $\pi^s\colon \cB \to \cP^s$. By Theorem~\ref{thm:bmp-classification}, every object of $\BMPe(\cP^I)$ is isomorphic to a finite direct sum of shifts of indecomposable BMP sheaves $\cE_p^I$ for $p \in W/W_I$, characterized by the following properties: $(\cE_p^I)^q = 0$ unless $q \le p$ (support condition); $(\cE_p^I)^p \cong R\{d_p\}$ (``perverse'' normalization).

Next, we recall Soergel-theoretic notions associated to $(W, \fh)$. The $W$ action on $\fh$ induces a $W$-action on $R$. For $s \in S$, let $R^s$ denote the $s$-invariants. Functors of restriction and induction
\[
 \theta^s_*\colon R\lh\gmod\lh R \to R\lh\gmod\lh R^s, \quad \theta^{s*}\colon R\lh\gmod\lh R^s \to R\lh\gmod\lh R
\]
satisfy the adjunctions
\begin{equation} \label{eq:theta-adj}
 \theta^{s*} \dashv \theta^s_* \dashv \theta^{s*}\{2\}
\end{equation}
(see~\cite[Proposition~5.10]{Soe07}).

Let $\SBim^I$ be the full subcategory of $R\lh\gmod\lh R^I$ of (singular) Soergel bimodules~\cite[Definition~7.1]{Wil-ssbim}. By the classification theorem~\cite[Theorem~7.10]{Wil-ssbim}, every object of $\SBim^I$ is isomorphic to a finite direct sum of shifts of indecomposable Soergel bimodules $B_p^I$ for $p \in W/W_I$, again characterized by a support condition and a normalization.

\begin{prop} \label{prop:bmp-sbim}
 Let $s \in S$. The pullback $\pi^{s*}\colon \She(\cP^s) \to \She(\cB)$ (see Definition~\ref{def:pullback}) preserves BMP sheaves. Moreover, there are functors $F, F^s, \pi^s_*$ as in the following diagram
 \[
  \xymatrix{
   \BMPe(\cB) \ar@<-0.6ex>[d]_{\pi^s_*} \ar[r]^-{F} & \SBim \ar@<-0.6ex>[d]_{\theta^s_*} \\
   \BMPe(\cP^s) \ar@<-0.6ex>[u]_{\pi^{s*}} \ar[r]^-{F^s} & \SBim^s \ar@<-0.6ex>[u]_{\theta^{s*}}
  }
 \]
 satisfying the following conditions:
 \begin{enumerate}
  \item All functors commute with $\{1\}$ (up to natural equivalence);
  \item $F$ and $F^s$ are equivalences;
  \item $F^s \circ \pi^s_* \cong \theta^s_* \circ F$ and $F \circ \pi^{s*} \cong \theta^{s*} \circ F^s$;
  \item $F(\cE_w) \cong B_w$ for $w \in W$, and $F^s(\cE_p^s) \cong B_p^s$ for $p \in W/W_s$.
 \end{enumerate}
 It follows from (1)--(3) and \eqref{eq:theta-adj} that $\pi^{s*} \dashv \pi^s_* \dashv \pi^{s*}\{2\}$, so $\pi^s\colon \cB \to \cP^s$ is PSS-1 (see Definition~\ref{defn:pss}).
\end{prop}

To work with constructible mixed derived categories, we need the following result.
\begin{prop} \label{prop:bruhat-V}
 The moment graphs $\cB$ and $\cP^s$ ($s \in S$) satisfy the condition~{\bf (V)}.
\end{prop}

\subsection{Proof of Propositions~\ref{prop:bmp-sbim} and \ref{prop:bruhat-V}}
\label{ss:bmp-sbim-proof}
Since only the statements of these results will be used elsewhere, and since most of this already appears in~\cite{Fie-verma, Fie-coxeter}, we merely review Fiebig's results and indicate the necessary modifications. In particular, we will not review the necessary moment-graph- and Soergel-theoretic notions from~\cite{Fie-verma, Fie-coxeter, Soe07, Wil-ssbim}.

Let us get started. As with any moment graph, global sections gives a functor
\[
 \Gamma\colon \She(\cP^I) \to \cZ^I\lh\gmod
\]
to graded modules over the structure algebra $\cZ^I = \cZ(\cP^I)$. Since $\cP^I$ is quasi-finite~\cite[Lemma~3.2(b)]{Fie-verma}, by~\cite[Theorem~3.6]{Fie-verma} and the discussion in \cite[\S3.7]{Fie-verma}, $\Gamma$ is fully faithful on the full subcategory of objects ``generated by global sections.'' This gives a category $\cC^I$ that will be viewed as a full subcategory of both $\She(\cP^I)$ and $\cZ^I\lh\gmod$.

The latter viewpoint relates more directly to Soergel bimodules. The following lemma gathers several statements proved in~\cite[\S5]{Fie-coxeter}.
\begin{lem} \label{lem:Zmod-RmodR}
 There are functors $F, F^s, \pi^s_*$, and $\pi^{s*}$ as in the following diagram
 \[
  \xymatrix{
   \cZ\lh\gmod \ar@<-0.6ex>[d]_{\pi^s_*} \ar[r]^-{F} & R\lh\gmod\lh R \ar@<-0.6ex>[d]_{\theta^s_*} \\
   \cZ^s\lh\gmod \ar@<-0.6ex>[u]_{\pi^{s*}} \ar[r]^-{F^s} & R\lh\gmod\lh R^s \ar@<-0.6ex>[u]_{\theta^{s*}}
  }
 \]
 satisfying the following conditions:
 \begin{enumerate}
  \item All functors commute with $\{1\}$;
  \item $F^s \circ \pi^s_* \cong \theta^s_* \circ F$, $F \circ \pi^{s*} \cong \theta^{s*} \circ F^s$;
  \item $\pi^{s*} \dashv \pi^s_* \dashv \pi^{s*}\{2\}$ and $\theta^{s*} \dashv \theta^s_* \dashv \theta^{s*}\{2\}$.
 \end{enumerate}
\end{lem}
At this stage, the functor $\pi^{s*}$ is unrelated to the pullback $\pi^{s*}\colon \She(\cP^s) \to \She(\cB)$.

To relate this to BMP sheaves, we need to introduce more categories. Let $\Cref^I \subset \cZ^I\lh\gmod$ be the full subcategory of ``reflexive'' objects, and let $\cV^I \subset \Cref^I$ be the full subcategory of objects that admit a Verma flag~\cite[\S4]{Fie-verma}. Since $(W, \fh)$ is reflection faithful, $\cB$ (resp.~$\cP^s$) is GKM~\cite[\S4.4]{Fie-verma}. It follows by~\cite[Proposition~4.6]{Fie-verma} that $\Cref^I \subset \cC^I$, so we have containments
\[
 \cV^I \subset \Cref^I \subset \cC^I.
\]
On the Soergel side, let $\cF_\nabla^I \subset R\lh\gmod\lh R^I$ be the full subcategory of objects that admit a nabla flag~\cite[Definition~6.1]{Wil-ssbim}.
\begin{prop} \label{prop:verma-nabla}
 The functor $F^I$ restricts to an equivalence
 \[
  F^I\colon \cV^I \simto \cF_\nabla^I.
 \]
\end{prop}
\begin{proof}
 For the regular case ($I = \emptyset$), this is~\cite[Theorem~4.3]{Fie-coxeter}. It is straightforward to extend Fiebig's proof to the case $I = \{s\}$, $s \in S$, by replacing the results used from regular Soergel theory by their singular analogues from~\cite{Wil-ssbim}. In particular, results on extensions of singular standard modules may be reduced to the regular case using~\cite[Theorem~4.10]{Wil-ssbim}; see~\cite[Lemmas~6.8--6.10]{Wil-ssbim} and the discussion preceeding these lemmas. One also uses the singular ``hin-und-her'' lemma~\cite[Lemma~6.2]{Wil-ssbim}.
\end{proof}

By~\cite[Theorem~6.4]{Wil-ssbim}, $\theta^s_*$ and $\theta^{s*}$ preserve objects that admit a nabla flag. It follows from Proposition~\ref{prop:verma-nabla} that $\pi^s_*$ and $\pi^{s*}$ preserve objects that admit a Verma flag. (This is also proved directly in~\cite{Fie-coxeter}: see the paragraph at the end of \S5.3 and Proposition~5.5.) Thus we have the diagram
\[
 \xymatrix{
  \cV \ar@<-0.6ex>[d]_{\pi^s_*} \ar[r]^-{F}_-{\sim} & \cF_\nabla \ar@<-0.6ex>[d]_{\theta^s_*} \\
  \cV^s \ar@<-0.6ex>[u]_{\pi^{s*}} \ar[r]^-{F^s}_-{\sim} & \cF_\nabla^s \ar@<-0.6ex>[u]_{\theta^{s*}}
 }
\]
where both squares commute (up to natural isomorphism). Moreover, for objects in $\cV^s$, $\pi^{s*}$ agrees with the pullback~\cite[Lemma~5.4]{Fie-coxeter}.

It follows from~\cite[Proposition~5.3]{Fie-coxeter} that $\pi^s_*$ and $\pi^{s*}$ also preserve reflexive objects:
\[
 \pi^s_*\colon \Cref \to \Cref^s, \quad \pi^{s*}\colon \Cref^s \to \Cref.
\]
Moreover, there is an exact structure on $\Cref^I$~\cite[\S4.1]{Fie-verma} so that the projective objects are precisely $\BMPe(\cP^I)$~\cite[Theorem~5.2]{Fie-verma}, and $\pi^s_*$ and $\pi^{s*}$ are exact~\cite[Proposition~5.6]{Fie-coxeter} (this is stated for $\cV^I$, but the proof also works for $\Cref^I$). Thus, by Lemma~\ref{lem:Zmod-RmodR}(3), each of $\pi^s_*$ and $\pi^{s*}$ is left adjoint to an exact functor, hence preserves projectives:
\[
 \pi^s_*\colon \BMPe(\cB) \to \BMPe(\cP^s), \quad \pi^{s*}\colon \BMPe(\cP^s) \to \BMPe(\cB).
\]

A support argument together with the classification of BMP sheaves now shows that $\BMPe(\cB)$ and $\BMPe(\cP^s)$ may be characterized as the smallest strictly full subcategories of $\Cref$ and $\Cref^s$ so that $\cE_e \in \BMPe(\cB)$ and that are closed under $\pi^s_*$, $\pi^{s*}$, finite direct sum, shift, and direct summand (cf.~\cite[proof of Theorem~6.1]{Fie-coxeter}). In particular, $\BMPe(\cB^I) \subset \cV^I$, which proves Proposition~\ref{prop:bruhat-V}.

Since $\SBim$ and $\SBim^s$ admit a similar characterization, it follows that the equivalences of Proposition~\ref{prop:verma-nabla} restrict to equivalences
\[
 F^I\colon \BMPe(\cP^I) \simto \SBim^I.
\]
This proves Proposition~\ref{prop:bmp-sbim}(1)--(3). Since $F^I$ is compatible with two notion of support in each category, (4) follows from the characterization of $\cE_p^I$ and $B_p^I$. This concludes the proof of Proposition~\ref{prop:bmp-sbim}.

\subsection{Convolution and perverse t-structure}
By transfering the tensor product $\otimes_R$ via the equivalences of Proposition~\ref{prop:bmp-sbim}, we obtain \emph{convolution functors}
\begin{align*}
 (-) \ast (-) \colon \BMPe(\cB) \times \BMPe(\cP^I) &\to \BMPe(\cP^I), \\
 (-) \ast (-) \colon \BMPc(\cB) \times \BMPe(\cP^I) &\to \BMPc(\cP^I),
\end{align*}
making $\BMPe(\cB)$ into a monoidal category with unit object $\cE_1$, with left module category $\BMPe(\cP^I)$ and right module category $\BMPc(\cB)$. Passing to the bounded homotopy category, we obtain
\begin{align*}
 (-) \ast (-) \colon \Dmixe(\cB) \times \Dmixe(\cP^I) &\to \Dmixe(\cP^I), \\
 (-) \ast (-) \colon \Dmixc(\cB) \times \Dmixe(\cP^I) &\to \Dmixc(\cP^I).
\end{align*}
Let $\cF \in \Dmixe(\cB)$ and $\cG \in \Dmixe(\cP^I)$. There are functorial isomorphisms
\begin{align}
 \For(\cF \ast \cG) &\cong \For(\cF) \ast \cG, \\
 \pi^*_s(\cF \ast \cG) &\cong \cF \ast \pi^{s*}(\cG) \quad \text{ if } I = \{s\}.
\end{align}
Each claim follow from the corresponding statement on the Soergel side.

By the discussion in~\S\ref{ss:constructible}, Proposition~\ref{prop:bruhat-V} implies that the theory of~\S\ref{s:recollement}--\ref{s:mixed-perverse-sheaves} applies to both $\Dmixe(\cP^I)$ and $\Dmixc(\cP^I)$. In particular, it makes sense to talk about standard sheaves $\Delta_p$, costandard sheaves $\nabla_p$, the categories of mixed perverse sheaves $\Pmixe(\cP^I)$, $\Pmixc(\cP^I)$, and simples $\IC_p$.

Let $s \in S$ and $x \in W$ be such that $xs > x$. Let $\overline{x} \in W/\{1, s\}$ be the image of $x$. Since $\pi^s\colon \cB \to \cP^s$ is PSS-$1$ (Proposition~\ref{prop:bmp-sbim}), by Lemma~\ref{lem:pss-delta-nabla} we have
\begin{align}\label{eqn:pi-delta-nabla}
 \pi^s_* \Delta_{xs} \cong \Delta_{\overline{x}}\{-1\}
 \quad\text{and}\quad
 \pi^s_* \Delta_x \cong \Delta_{\overline{x}}.
\end{align}
As in~\cite[Corollary~3.9--3.10]{AR-II}, it then follows from Proposition~4.4 that $\pi^s_*\{1\}$ is right t-exact, $\pi^s_*\{-1\}$ is left t-exact, $\pi^{s\ast}\{1\}$ is t-exact, and $\pi^{s\ast}\IC_{\overline{x}} \cong \IC_{xs}$.

\begin{rmk}
 As pointed out in~\cite[Remark 6.2]{EW-hodge}, if $(W, \fh)$ satisfies Soergel's conjecture, one obtains a t-structure on $\Kb\SBim$ by defining the heart to consist of complexes that are homotopy equivalent to complexes $\cF$ such that for all $i \in \Z$, $\cF^i$ is a direct sum of indecomposable Soergel bimodules $B_w\{i\}$. Via Proposition~\ref{prop:bmp-sbim}, this induces a t-structure on $\Dmixe(\cB)$. One can show using facts about exceptional collections (see e.g.~\cite[Section~2.1]{Bez}) that this t-structure agrees with the one defined using recollement.
\end{rmk}

\subsection{$\mathsf{A}_1$ computations}
\label{ss:a1-computation}
Let $s \in S$. Consider the closed sub-moment graph $Y_1^s = \{1, s\}$ of $\cB$, with unique edge $E\colon 1\text{---}s$ labeled $\alpha_s$. We will depict a moment graph sheaf $\cF$ on $\cB$ supported on $Y_1^s$ as follows:
\[
 \cF = \begin{array}{c} \xymatrix@R=13pt{\cF^s \ar[d] \\ \cF^E \\ \cF^1 \ar[u]} \end{array}
\]
It is easy to see that
\[
 \cE_1 = \ubk_{\cB_1} = \begin{array}{c} \xymatrix@R=13pt{0 \ar[d] \\ 0 \\ R \ar[u]} \end{array},
 \qquad
 \cE_s = \ubk_{Y^s_1}\{1\} := \begin{array}{c} \xymatrix@R=13pt{R\{1\} \ar[d]^q \\ (R/\alpha_s)\{1\} \\ R\{1\} \ar[u]_q} \end{array},
\]
where $q$ is the natural quotient map.

Let $\eta_s\colon \cE_1\{-1\} \to \cE_s$ and $\epsilon_s\colon \cE_s \to \cE_1\{1\}$ be the following morphisms:
\[
 \begin{array}{c}\xymatrix@R=13pt{
  0 \ar[d] \ar[r] & R\{1\} \ar[d]^q \\
  0 \ar[r] & (R/\alpha_s)\{1\} \\
  R\{-1\} \ar[u] \ar[r]^{\alpha_s} & R\{1\} \ar[u]_q.
 }\end{array}, \qquad
 \begin{array}{c}\xymatrix@R=13pt{
  R\{1\} \ar[d]^q \ar[r] & 0 \ar[d] \\
  (R/\alpha_s)\{1\} \ar[r] & 0 \\
  R\{1\} \ar[u]_q \ar[r]^1 & R\{1\} \ar[u].
 }\end{array}.
\]
We have
\[
 \Hom_{\BMPe(\cB)}(\cE_1\{-1\}, \cE_s) = \bk \cdot \eta_s, \qquad \Hom_{\BMPe(\cB)}(\cE_s, \cE_1\{1\}) = \bk \cdot \epsilon_s.
\]

Consider the inclusions
\[
 \cB_1 \overset{i}{\hookrightarrow} Y^s_1 \overset{j}{\hookleftarrow} \cB_s.
\]
Recall the naive sheaf functors from~\S\ref{ss:naive-functors}. The unit of $i_* \dashv i^{[!]}$ (resp.~the counit of $i^{[*]} \dashv i_*$) at $\cE_s$ equals
\[
 \eta_s\colon \cE_1\{-1\} = i_!i^{[!]}\cE_s \to \cE_s \qquad (\text{resp.~}\epsilon_s\colon \cE_s \to i_*i^{[*]}\cE_s = \cE_1\{1\}).
\]
From the definition of $j_!$ (resp.~$j_*$) in~\S\ref{ss:recollement}, we find that $\Delta_s$ (resp.~$\nabla_s$) in $\Dmix(\cB)$ is the image of the complex
\[
 \cE_1\{-1\} \xrightarrow{\eta_s} \cE_s \qquad (\text{resp.~}\cE_s \xrightarrow{\epsilon_s} \cE_1\{1\}),
\]
where $\cE_s$ is in degree 0.

\subsection{Soergel-theoretic graded category $\cO$}
\label{ss:consequences}
With the results above in place, we can prove analogues of many of the results from~\cite{AR-II} on flag varieties.

\begin{lem}[cf.~\cite{AR-II}, Lemma~4.1]
 Let $s \in S$ and $x \in W$ with $xs > x$. Consider the morphisms
 \begin{align*}
  \eta &\colon \Delta_{xs} \to \pi^{s!}\pi^s_! \Delta_{xs} \simto \pi^{s!}\Delta_{\overline{x}}\{ -1 \} \simto \pi^{s*}\pi^s_* \Delta_x\{1\} \\
  \varepsilon &\colon \pi^{s*}\pi^s_* \Delta_x\{1\} \to \Delta_x\{1\},
 \end{align*}
 where the isomorphisms in $\eta$ are defined by \eqref{eqn:pi-delta-nabla}, and the other two morphisms are defined by adjunction. Then there exists a distinguished triangle
 \[
  \Delta_{xs} \xrightarrow{\eta} \pi^{s*}\pi^s_* \Delta_x\{1\} \xrightarrow{\varepsilon} \Delta_x\{1\} \to
 \]
in $\Dmixe(\cB)$.
\end{lem}
\begin{proof}
 Let $Y^s_x = (\pi^s)^{-1}(\cP^s_{\overline{x}})$, so $\cV(Y^s_x) = \{x, xs\}$. Consider the following diagram:
 \[
  \xymatrix{
   Y^s_x \ar[d]_{\pi^s} \ar@{^(->}[r]^h & \cB \ar[d]^{\pi^s} \\
   \cP^s_{\overline{x}} \ar@{^(->}[r]_{i_{\overline{x}}} & \cP^s
  }
 \]
 By Proposition~\ref{prop:pss}, $\pi^{s*}\colon \Dmixe(\cP^s_{\overline{x}}) \to \Dmixe(Y^s_x)$ admits a biadjoint $\pi^s_*$ up to shift. Define the functor
 \begin{align*}
  \pi^s_{(*)}\colon \She(Y^s_x) &\to \She(\cP^s_{\overline{x}}) \\
  (\pi^s_{(*)}\cF)^{\overline{x}} &= \Gamma(Y^s_x, \cF),
 \end{align*}
 where the $R$-algebra structure on $\Gamma(Y^s_x, \cF)$ is coordinate-wise multiplication by the constant section. It is easy to show that $\pi^{s*} \dashv \pi^s_{(*)}$ and that $\pi^s_{(*)}$ preserves BMP sheaves, so $\pi^s_* \cong \pi^s_{(*)}$.

 Consider the inclusions $i\colon \cB_x \hookrightarrow Y^s_x$ and $j\colon \cB_{xs} \hookrightarrow Y^s_x$. As in~\S\ref{ss:a1-computation}, applying the functorial distinguished triangle $j_!j^! \to \id \to i_*i^* \to$ to the BMP sheaf $\ubk_{Y^s_x}\{d_{xs}\}$, we obtain a triangle
 \[
  \Delta_{Y^s_x, xs} \to \ubk_{Y^s_x}\{d_{xs}\} \to \Delta_{Y^s_x, x}\{1\} \to.
 \]
 Using the description $\pi^s_* \cong \pi^s_{(*)}$ above, the middle term can be identified with $\pi^{s*}\pi^s_*\Delta_x\{1\}$. Applying $h_!$ and using Proposition~\ref{prop:pss}, we obtain a distinguished triangle
 \[
  \Delta_{xs} \to \pi^{s*}\pi^s_*\Delta_x\{1\} \to \Delta_x\{1\} \to
 \]
 in $\Dmixe(\cB)$. Since
 \begin{multline*}
 \Hom_{\Dmixe(\cB)}(\pi^{s*}\pi^s_*\Delta_x\{1\}, \Delta_x\{1\}) \cong \Hom_{\Dmixe(\cB)}(\pi^s_*\Delta_x\{1\}, \pi^s_*\Delta_x\{1\}) \\
 \cong \Hom_{\Dmixe(\cP^s)}(\Delta_{\overline{x}}\{-1\}, \Delta_{\overline{x}}\{-1\})
 \cong \Hom_{\Dmixe(\cP^s_{\overline{x}})} (\ubk_{\cP^s_{\overline{x}}}, \ubk_{\cP^s_{\overline{x}}})
 \end{multline*}
is free of rank 1 over $\bk$, one can assume that the second morphism above is induced by adjunction, and similarly for the first morphism.
\end{proof}

\begin{prop}[cf.~\cite{AR-II}, Proposition~4.4] \label{prop:convo-facts}
 \begin{enumerate}
 \item If $\ell(yw) = \ell(y) + \ell(w)$, then we have isomorphisms 
  \[
   \Delta_{yw} \cong \Delta_y \ast \Delta_w, \qquad \nabla_{yw} \cong \nabla_y \ast \nabla_w
  \]
  in $\Dmixe(\cB)$.
 \item We have isomorphisms
  \[
   \Delta_w \ast \nabla_{w^{-1}} \cong \cE_1 \cong \nabla_{w^{-1}} \ast \Delta_w
  \]
  in $\Dmixe(\cB)$.
 \end{enumerate}
\end{prop}

\begin{prop}[cf.~\cite{AR-II}, Proposition~4.6] \label{prop:convo-exact}
 Let $w \in W$.
 \begin{enumerate}
 \item The functors 
  \[
  ({-}) \ast \nabla_w, \ \nabla_w \ast ({-})\colon \Dmixe(\cB) \to \Dmixe(\cB)
  \]
  are right t-exact with respect to the perverse t-structure.\label{it:convo-right-exact}
 \item The functors 
  \[
  ({-}) \ast \Delta_w, \ \Delta_w \ast ({-})\colon \Dmixe(\cB) \to \Dmixe(\cB)
  \]
  are left t-exact with respect to the perverse t-structure.\label{it:convo-left-exact}
 \end{enumerate}
 In particular, for any $w, y \in W$, $\nabla_y \ast \Delta_w$ and $\Delta_w \ast \nabla_y$ are perverse.
\end{prop}

\begin{thm}[cf.~\cite{AR-II}, Theorem~4.7] \label{thm:ghw}
 Let $I  = \emptyset$ or $I = \{s\}$, where $s \in S$. The (partial) Bruhat moment graph $\cP^I$ satisfies assumption~{\bf (A2)}. Hence by Proposition~\ref{prop:ghw-structure}, $\Pmixc(\cP^I)$ is a graded highest weight category.
\end{thm}

\begin{ex}
 Let $s \in S$. One easily checks from the definitions that $\cT_s \in \Pmixc(\cB)$ is represented by the image in $\Dmixc(\cB)$ of the three-term complex
 \[
  \cE_1\{-1\} \xrightarrow{\eta_s} \cE_s \xrightarrow{\epsilon_s} \cE_1\{1\},
 \]
 where $\cE_s$ is in degree 0, and $\eta_s$ and $\epsilon_s$ are the morphisms described in~\S\ref{ss:a1-computation}.
\end{ex}

The following result only makes sense when $W$ is finite. Let $w_0$ denote the longest element of $W$.
\begin{prop}[Geometric Ringel duality; cf.~\cite{AR-II}, Proposition~4.11] \label{prop:ringel}
 The functor
 \[
  R := ({-}) \ast \Delta_{w_0}\colon \Dmixc(\cB) \to \Dmixc(\cB)
 \]
 is a triangulated equivalence with quasi-inverse $({-}) \ast \nabla_{w_0}$. Moreover,
 \[
  R(\nabla_w) \cong \Delta_{ww_0}, \qquad R(\cT_w) \cong \cP_{ww_0}.
 \]
\end{prop}

\subsection{Koszulity}
\label{ss:koszul}
In this subsection, we discuss the Koszulity of $\Pmixc(\cB)$.

\begin{prop} \label{prop:soergel-conjecture}
The following are equivalent:
\begin{enumerate}
  \item for all $w \in W$, $\cE_w \cong \IC_w$;
  \item for all $x < w$, $j_x^*\cE_w$ is generated in degrees $< -\ell(x)$;
  \item for all $x < w$, $j_x^!\cE_w$ is generated in degrees $> -\ell(x)$;
  \item $(W, \fh)$ satisfies (the analogue of) Soergel's conjecture (see~\cite[Conjecture~3.16]{EW-soergel-calculus}).
 \end{enumerate}
\end{prop}
\begin{proof}
 The equivalence (1) $\Leftrightarrow$ (2)+(3) follows from the characterization of $\IC_w$ in~\cite[Corollaire~1.4.24]{BBD}, (2) $\Leftrightarrow$ (3) follows from the self-duality of indecomposable Soergel bimodules under the contravariant duality, and (2) $\Leftrightarrow$ (4) is proved in~\cite[Proposition 8.4]{Fie-coxeter}.
\end{proof}

Assume that $W$ is finite, so projective covers $\cP_w \twoheadrightarrow \IC_w$ exist in $\Pmixc(\cB)$. Consider the objects
\[
 \cP := \bigoplus_{w \in W} \cP_w, \qquad \cE := \bigoplus_{w \in W} \cE_w, \qquad \IC := \bigoplus_{w \in W} \IC_w
\]
in $\Dmixc(\cB)$. Define the graded $\bk$-algebras
\begin{align*}
 A^\proj_{W, \fh} := \left(\bigoplus_n \Hom(\cP, \cP\la n \ra)\right)^\opp, \qquad A^\parity_{W, \fh} := \bigoplus_n \Hom(\cE, \cE\{n\}).
\end{align*}
\begin{prop}
 Suppose that $(W, \fh)$ satisfies Soergel's conjecture. Then $A^\proj_{W, \fh}$ and $A^\parity_{W, \fh}$ are Koszul, and Koszul dual to each other.
\end{prop}
See~\cite[\S2]{BGS} for a discussion of a Koszul (graded) ring; in the proof below, we verify one of the equivalent conditions defining Koszulity. For a Koszul ring $B = \bigoplus_{i \ge 0} B_i$, its Koszul dual ring is $E(B) := \Ext_{B\lh\ngmod}^\bullet(B_0)$, where Ext is taken in ungraded $B$-modules.
\begin{proof}
 To ease the notation, we drop the subscript $W, \fh$.
 
 First, as in~\cite[Corollary~3.17]{AR-III}, by the second equivalence in Lemma~\ref{lem:tilt-perv-der-eq} and Proposition~\ref{prop:soergel-conjecture}, we have
 \begin{equation} \label{eq:koszul-ext-hom}
  \bigoplus_{n, m}\Ext_{\Pmixc(\cB)}^n(\IC, \IC\la m \ra) \cong \bigoplus_{n, m}\Hom_{\Dmixc(\cB)}^n(\cE, \cE[n+m]\{-m\}).
 \end{equation}
 Since $\cE$ is concentrated in cohomological degree 0, the Hom space in the last expression vanishes unless $m = -n$. Thus $\Pmixc(\cB)$ is a Koszul category in the sense of~\cite{AR-III}. Moreover, the last expression equals
 \[
  \bigoplus_n\Hom_{\Dmixc(\cB)}^n(\cE, \cE\{n\}) = A^\parity.
 \]
 
 By~\cite[Proposition~2.5]{AR-III}, $\Pmixc(\cB)$ Koszul implies that $A^\proj$ is positively graded. Now, by a standard argument, there is an equivalence
 \begin{align} \label{eqn:Pmix-gmod}
  \alpha := \bigoplus_{n \in \Z} \Hom(\cP, (-)\la n \ra)\colon \Pmixc(\cB) \simto A^\proj\lh\gmod^\fg
 \end{align}
 satisfying $\alpha \circ \la 1 \ra = \{1\} \circ \alpha$ and $\alpha(\IC) \cong (A^\proj)_0 = A^\proj/(A^\proj)_{>0}$. We therefore have a graded $\bk$-algebra isomorphism
 \[
  E(A^\proj) = \bigoplus_{n, m} \Ext_{A^\proj\lh\gmod}^n((A^\proj)_0, (A^\proj)_0\{m\})
  \cong \bigoplus_{n, m}\Ext_{\Pmixc(\cB)}^n(\IC, \IC\la m \ra).
 \]
 Since $A^\proj$ is positively graded with $(A^\proj)_0$ semisimple, the same vanishing as in~\eqref{eq:koszul-ext-hom} is equivalent to the Koszulity of $A^\proj$ \cite[Proposition 2.1.3]{BGS}. Putting everything together, the result is proved.
\end{proof}

\subsection{Application to Rouquier complexes}
\label{ss:rouquier-complexes}
We recall the definition of Rouquier complexes~\cite{Rou}. For $s \in S$, let $F_s$ (resp.~$E_s$) be the image in $\Kb\SBim$ of the complex
\[
 B_s \to R\{1\} \qquad (\text{resp.~} R\{-1\} \to B_s),
\]
where $B_s$ is in degree 0, and the differential is the unique nonzero map up to scalar. For $w \in W$, choose a reduced expression $w = s_1 \ldots s_l$, and define
\[
 F_w = F_{s_1} \cdots F_{s_l}, \qquad E_w = E_{s_1} \cdots E_{s_l}.
\]
Rouquier showed that $F_s, E_s$ satisfy the braid relations up to homotopy equivalence, so that $F_w$ and $E_w$ are well-defined up to isomorphism as objects in $\Kb\SBim$.

Let $\Phi\colon \BMPe(\cB) \simto \SBim$ be the equivalence of Proposition~\ref{prop:bmp-sbim}.
\begin{thm} \label{thm:std-rouquier}
 For all $w \in W$, we have
 \[
  \Kb\Phi(\Delta_w) \cong F_w, \qquad \Kb\Phi(\nabla_w) \cong E_w.
 \]
\end{thm}
\begin{proof}
 For $w \in S$, this follows from the computation in~\S\ref{ss:a1-computation} and Proposition~\ref{prop:bmp-sbim}. The general case follows by induction on $\ell(w)$ using Proposition~\ref{prop:convo-facts}.
\end{proof}
In fact, this gives a new proof that $F_s, E_s$ satisfy the braid relations.

\begin{rmk}
 If one unravels the intuition for the mixed derived category, this result says that Rouquier complexes describe the subquotients in the weight filtrations of standard and costandard sheaves on the flag variety associated to $(W, \fh)$. For $W$ a Weyl group, this interpretation has been described in~\cite[\S3.3]{WW}.
\end{rmk}
 
\begin{cor}[Rouquier's formula~\cite{LW}] \label{cor:rouquier-formula}
 Let $x, w \in W$. We have
  \[
  \bigoplus_{n \in \Z} \Hom_{\Kb\SBim} (F_x, E_w \{n\} [i]) \cong
   \begin{cases}
    R & \text{if $x = w$ and $i = 0$;} \\
    0   & \text{otherwise,}
   \end{cases}
 \]
 as graded $R$-modules.
\end{cor}
\begin{proof}
 By Theorem~\ref{thm:std-rouquier}, this follows from $\Kb\Phi$ applied to Lemma~\ref{lem:vanishing-delta-nabla}.
\end{proof}

\bibliographystyle{alpha}
\bibliography{refs.bib}

\end{document}